\pdfoutput=1
\documentclass[a4paper, 11pt, abstracton, titlepage, oneside]{scrartcl}
\makeatletter

\usepackage[T1]{fontenc}
\usepackage[utf8]{inputenc}
\usepackage[onehalfspacing]{setspace}
\usepackage[headsepline]{scrlayer-scrpage}
\clearscrheadfoot % reset standard format
\usepackage{layout}
\usepackage{geometry}
\usepackage{adjustbox}
\usepackage{authblk}
\usepackage[titletoc,title]{appendix}

\usepackage[hyphens]{url}
\usepackage{color}

\usepackage{amsmath,amssymb,amsfonts,amsthm, mathtools, bm}
\usepackage{thmtools, thm-restate}
\usepackage{nicefrac}
\usepackage{array}

\usepackage[normal]{threeparttable}
\usepackage{threeparttablex}
\usepackage{tabularx}
\usepackage{longtable}
\usepackage{booktabs}
\usepackage{colortbl}
\usepackage{multirow}
\usepackage{makecell}

\usepackage[acronym]{glossaries}

\usepackage{subcaption}
\usepackage{floatrow}
\usepackage{graphicx}
\usepackage{placeins}

\usepackage{algorithm}
\usepackage{algpseudocode}

\usepackage{tikz}
\usetikzlibrary{arrows.meta}
\usetikzlibrary{math}
\usetikzlibrary{shapes}
\usepackage{pgfplots}
\usepgfplotslibrary{groupplots}
\usepgfplotslibrary{dateplot}
\usepackage{tikzscale}

\usepackage{enumerate}
\usepackage{multicol}
\usepackage[author=Patrick]{pdfcomment}
\usepackage{xparse}
\usepackage{twoopt}
\usepackage{etoolbox}

\usepackage{eurosym}
\usepackage{ellipsis}
\usepackage{natbib}
\usepackage{paralist}
\usepackage{ulem}
\AtBeginDocument{
	\newgeometry{
		textheight=9in,
		textwidth=6.5in,
		top=1in,
		headheight=14pt,
		headsep=25pt,
		footskip=30pt
	}
}
\newsavebox{\abstractbox}
\renewenvironment{abstract}
{\begin{lrbox}{0}\begin{minipage}{\textwidth}
			\begin{center}\normalfont\sectfont\abstractname\end{center}\quotation}
		{\endquotation\end{minipage}\end{lrbox}%
	\global\setbox\abstractbox=\box0 }

\makeatletter
\expandafter\patchcmd\csname\string\maketitle\endcsname
{\vskip\z@\@plus3fill}
{\vskip\z@\@plus2fill\box\abstractbox\vskip\z@\@plus1fill}
{}{}
\makeatother

\addtokomafont{captionlabel}{\footnotesize\bfseries} % style for caption labels
\addtokomafont{caption}{\footnotesize\bfseries} % and the caption
\bibpunct[, ]{(}{)}{,}{a}{}{,}%
\newtheorem{definition}{Definition}[section]
\newtheorem{theorem}{Theorem}[section]

\newenvironment{remark}{\textit{Remark:}}

\counterwithin*{equation}{section}

\newenvironment{myfont}{\fontfamily{ccr}\selectfont}{\par}
\DeclareTextFontCommand{\textmyfont}{\myfont}

\AtBeginDocument{%
	\abovedisplayskip=7.5pt plus 0pt minus 0pt
	\abovedisplayshortskip=7.5pt plus 0pt minus 0pt
	\belowdisplayskip=7.5pt plus 0pt minus 0pt
	\belowdisplayshortskip=7.5pt plus 0pt minus 0pt
}
\newcolumntype{L}[1]{>{\raggedright\let\newline\\\arraybackslash\hspace{0pt}}p{#1}}
\newcolumntype{C}[1]{>{\centering\let\newline\\\arraybackslash\hspace{0pt}}p{#1}}
\newcolumntype{R}[1]{>{\raggedleft\let\newline\\\arraybackslash\hspace{0pt}}p{#1}}
\renewcommand{\emph}[1]{\textit{#1}}
\graphicspath{{Figures/}}

\begin{document}
\emergencystretch 3em
\title{\large Optimizing intermodal transportation networks at scale via column generation}

% and the authors
\author[1]{\normalsize Benedikt Lienkamp}
\author[2]{\normalsize Maximilian Schiffer}
\affil{\small 
	TUM School of Management, Technical University of Munich, 80333 Munich, Germany
	
	\scriptsize benedikt.lienkamp@tum.de
	
	\small
	\textsuperscript{2}TUM School of Management \& Munich Data Science Institute,
	
	Technical University of Munich, 80333 Munich, Germany
	
	\scriptsize schiffer@tum.de}

% if you like - a date
\date{}

% in case you have a headline - otherwise outcomment
\lehead{\pagemark}
%\rehead{\normalfont\scriptsize\textbf{Schiffer et. al.:} \textit{Perspectives for electric commercial vehicles}}
%\lohead{\normalfont\scriptsize\textbf{Schiffer et. al.:} \textit{Perspectives for electric commercial vehicles}}
\rohead{\pagemark}

% finally your abstract
\begin{abstract}
\begin{singlespace}
{\small\noindent In light of the need for design and analysis of intermodal transportation systems, we propose an algorithmic framework to determine the system optimum of an intermodal transportation system. To this end, we model an intermodal transportation system by combining two core principles of network optimization -- layered-graph structures and (partially) time-expanded networks -- to formulate our problem on a graph that allows us to implicitly encode problem specific constraints related to intermodality. This enables us to solve a standard integer minimum-cost multi-commodity flow problem to obtain the system optimum for an intermodal transportation system. To solve this integer minimum-cost multi-commodity flow problem efficiently, we present a column generation approach to find continuous minimum-cost multi-commodity flow solutions, which we combine with a price-and-branch procedure to obtain integer solutions. To speed up our column generation, we further develop a pricing filter and an admissible distance approximation to utilize the $A^*$ algorithm for solving the pricing problems. We show the efficiency of our framework by applying it to a real-world case study for the city of Munich, where we solve instances with up to 56,295 passengers to optimality, and show that the computation time of our algorithm can be reduced by up to 60\% through the use of our pricing filter and by up to additional 90\% through the use of the $A^*$-based pricing algorithm.\\
\smallskip}
{\footnotesize\noindent \textbf{Keywords:} column generation; multi-commodity flow; intermodal transportation}
\end{singlespace}
\end{abstract}

% don't forget to make the tile
\maketitle
% and your chapters
\section{Introduction}
\label{sec:introduction}
Multi-commodity network flow (MCNF) problems have been vividly studied during the last decades and have been applied to various real-world domains. Among others, they have been extensively used to model communication and transportation systems. Here, exact algorithms have been developed to solve medium scale problems, e.g., for message routing \citep{barnhart2000using} and heuristic algorithms have been proposed to quickly obtain good quality solutions for large scale problems, e.g., for general linear minimum-cost MCNF problems \citep{salimifard2022multicommodity}. This state of the art has been used during the last years to study transportation systems as it suffices for medium scale problem sizes or mesoscopic system analyses.

However, recent developments in the transportation sector, e.g., the advent of mobility as a service, intermodal trips, autonomous vehicles, ride-hailing, and ride-pooling require in-depth analyses of the resulting complex transportation systems at scale. In this context, an optimization-based analysis of the transportation system’s optimum is of interest, as future intermodal transportation systems with autonomous elements allow for more control and are thus amenable for the integration of optimal routing decisions. Determining system optimal solutions in such a context requires to solve MCNF problems that significantly exceed problem sizes that have so far been solved to optimality. Additionally, the intermodal structure of future transportation systems needs to be considered.

Against this background, we revisit the state of the art on solving MCNF problems for transportation systems. We aim to develop a column generation-based algorithm to find system optima of intermodal transportation networks and improve upon the current state of the art in terms of computational tractability and scalability. In the remainder of this section, we first review related literature, before we define our contributions and the paper's organization.

\subsection{Related Literature}
MCNF models have been used to model a variety of problems, amongst others in communication (\citealp{ouorou2000survey}; \citealp{lin2001study}; \citealp{wagner2007multi}), logistics (\citealp{erera2005global}; \citealp{ALKHAYYAL2007106}; \citealp{psaraftis2011multi}), and transportation systems (\citealp{ayar2012intermodal}; \citealp{paraskevopoulos2016congested}; \citealp{zhang2019multi}). Besides these research fields, MCNF models also find application in financial flows \citep{yang2015bitcoin}, evacuation planning \citep{pyakurel2017continuous}, production services \citep{atamturk2007two}, and traffic control \citep{bertsimas2000traffic}, which shows the great versatility of MCNF models. Various heuristics, approximations, and exact solution methods have been used to solve different variants of MCNF problems.

The most prominent heuristic methods to solve MCNF problems are genetic algorithms (\citealp{khouja1998use}; \citealp{alavidoost2018bi}), simulated annealing (\citealp{yaghini2012simplex}; \citealp{moshref2016customer}), and tabu search (\citealp{ghamlouche2003cycle}; \citealp{crainic2006first}).
Various approximation methods exist for different types of network flow formulations. Lagrangian-based methods proved to be very efficient at solving MCNF problems and find their application in solving multi-commodity capacitated fixed charge network design problems \citep{crainic2001bundle} and minimum-cost MCNF problems \citep{retvdri2004novel}.

The two most prominent exact approaches to solve MCNF problems are branching and decomposition methods. Branching problems build on the idea of branch-and-bound (B\&B) \citep{lawler1966branch} and have been used to solve integer MCNF problems \citep[see, e.g.,][]{brunetta2000polyhedral}. Decomposition methods were developed by \cite{dantzig1960decomposition} and used, for example, by \cite{barnhart1994column} to solve message routing problems, formulating them as minimum-cost MCNF problems and solving them via column generation (CG). 
\cite{barnhart2000using} utilize the branch-and-price (B\&P) algorithm \citep{barnhart1998branch}, a combination of CG and B\&B, to solve integer MCNF problems. Here, the authors applied bounds provided by solving linear programs, using column-and-cut generation at nodes of the branch-and-bound tree in combination with cuts, generated at each node in the B\&B tree, to mitigate symmetry effects.

In general, MCNF models have recently been used to model transportation systems, e.g., autonomous Mobility-on-Demand (AMoD) systems in congested road networks \citep{rossi2018routing} and to study the interaction of AMoD with the public transportation system \citep{salazar2019intermodal} on a mesoscopic level. We refer to \cite{salimifard2022multicommodity} for an extensive overview of applications and solution methods for MCNF problems.

Dynamic flows add a time dimension to static network flow problems and allow flow values on arcs to change over time. In dynamic network flow models for traffic optimization, the size of the underlying graph can be a limiting factor for the model's granularity or time horizon. To address the problem of exponential growth in dynamic networks, \cite{Boland2017}~introduced the concept of partially time-expanded networks to service network design problems. In partially time-expanded networks, not all timesteps are included in the network graph, which allows for iterative refinement to obtain a reduced graph size. Besides a reduced size of the underlying graph, multiple methods exist to speed up computation time in large network flow models. For an extensive overview of dynamic network flow problems, we refer to \cite{skutella2009introduction}.

In summary, various optimization approaches exist to solve MCNF problems for different problem settings.
So far, using an MCNF model to optimize intermodal passenger transport has only been done on a mesoscopic level \citep{salazar2019intermodal}. Analysis of passenger transport in large transportation networks on a microscopic level is often done via simulation \citep{ziemke2019matsim}, which is not amenable for optimization. To the best of the authors' knowledge, there exists no optimization framework that is capable of determining the system optimum of an intermodal transportation system on a microscopic level.

\subsection{Contribution}
To close the research gap outlined above, we propose an algorithmic framework to determine the system optimum of an intermodal transportation system. Specifically our contribution is four-fold: First, we model an intermodal transportation system by combining two core principles of network optimization, layered-graph structures and (partially) time-expanded networks, which allows us to implicitly encode problem specific constraints related to intermodality. This enables us to solve a standard integer minimum-cost MCNF problem to determine the system optimum for an intermodal transportation system. Second, to solve the integer minimum-cost MCNF problem efficiently, we present a CG approach to find continuous minimum-cost MCNF solutions, which we combine with a price-and-branch (P\&B) procedure to obtain integer solutions. To speed up our CG approach, we develop a pricing filter and an admissible distance approximation to utilize the $A^*$ algorithm for solving the pricing problems. Third, we show the efficiency of our framework by applying it to a real-world case study for the city of Munich, where we solve instances with up to 56,295 passengers to optimality, and show that the computation time of our algorithm can be reduced by up to 60\% through the use of our pricing filter and by up to additional 90\% through the use of the $A^*$-based pricing algorithm. Fourth, we open this algorithmic framework as open source code for further research.

\subsection{Organization}
The remainder of this paper is structured as follows. We specify our problem setting in Section \ref{sct: problem setting} and develop our methodology in Section \ref{sec:methododlogy}. In Section \ref{sec: case study} we describe a case study for the public transportation system for the city of Munich. We use this case study in Section \ref{sec: results} to present numerical results that show the efficiency of our algorithmic framework. Section \ref{sec: conclusion} concludes this paper by summarizing its main findings.
\section{Problem Setting}
\label{sct: problem setting}
We study optimal passenger routing in a large intermodal capacitated transportation system with fixed vehicle routes, e.g., bus, subway, and tram lines. An instance of our problem consists of a schedule of fixed vehicle routes in a transportation system and a set of passengers. The vehicle schedules contain information about the stops of each route, arriving times at its stops, and the capacity of the operating vehicle. Each passenger is associated with a transportation request, such that information about a trip's origin and destination coordinates and the departure time is available. In this setting, we aim to find paths for all passengers, such that the sum over all passenger travel times is minimal and all vehicle capacity constraints are kept. We formalize the underlying planning problem as follows.

\textit{Notation:} We consider a set of passengers $\mathcal{P} = \{1, 2, ..., P\}$ and a set of vehicle route schedules $\mathcal{R} = \{1, 2, ..., R\}$ in the transportation system. Each passenger $p \in \mathcal{P}$ is associated with a request tuple $\zeta_p = (o_p, d_p, \delta_p)$ comprising an origin coordinate $o_p$, a destination coordinate $d_p$ and a departure time $\delta_p$, which is the timestep in which a passenger wants to begin its trip. Every vehicle route schedule $r \in \mathcal{R}$ contains information about the stops $\mathcal{S}_r$ of the route, the arriving times $T_r^{\alpha}(s)$ at stops $s\in \mathcal{S}_r$, and the capacity $\gamma_r$ of the vehicle operating the route. Implicitly, the arrival times indicate in which order a route's stops are visited. Furthermore, we define $\mathcal{S} = \bigcup_{r \in \mathcal{R}} S_r$ as the set of all stops of the transportation system and $T(s) = \bigcup_{r \in \mathcal{R}} T_r^{\alpha}(s)$ as the set of all timestep in which vehicles arrive at stop $s \in \mathcal{S}$.

\textit{Solution:} A solution to a problem instance is a set of paths $\boldsymbol{\psi} = (\phi_{1}, \phi_{2}, ..., \phi_{P})$ with one path for each passenger $p \in \mathcal{P}$. Every path $\phi_p$ is either a list of tuples $[(o_p, \delta_p), (s_p^1, t_p^1), (s_p^2, t_p^2) ...,(d_p,t_p^{n_p})]$, or an empty set $\emptyset$ if there does not exist a feasible path for passenger $p$. Here for $i~\in~[1,..., n_p -1]$, $s_p^i \in \mathcal{S}$ is a stop in the transportation system, $t_p^i \in T(s_p^i)$ is a timestep in which a vehicle arrives at stop $s_p^i$, and $t_p^{n_p} > 0$ is the timestep at which passenger $p$ arrives at its destination coordinate.

\textit{Objective:} Our objective is to minimize the sum of travel times for all passengers $p \in \mathcal{P}$
\begin{equation}
\min c(\boldsymbol{\psi}) = \min \sum_{p \in \mathcal{P}} c(\phi_{p})
\end{equation}

Here, $c(\phi_{p}) = t_p^{n_p} - \delta_p$ is the travel time of passenger $p \in \mathcal{P}$ when using path $\phi_{p}$. If no feasible path was found for pasenger $p$, i.e., $\phi_p = \emptyset$, we assign $c(\phi_{p}) = \rho$, with $\rho$ being a predefined penalty.

\textit{Constraints:} A solution $\boldsymbol{\psi}$ is feasible if the following constraints hold:
\begin{itemize}
\item[(i)] The capacity $\gamma_r$ of a vehicle driving route $r \in \mathcal{R}$ is greater than the number of passengers using the route at each point in time.
\item[(ii)] For each tuple pair $(s_p^i, t_p^i), (s_p^{i+1}, t_p^{i+1})$, $i \in [1,..., n_p - 1]$ in $\phi_{p}$ with $p\in \mathcal{P}$, there 
	\begin{itemize}
		\item[(a)] either exists a route $r \in \mathcal{R}$ where the vehicle departs from $s_p^i$ in timestep $t_p^i$ to 							arrive at $s_p^{i+1}$ in $t_p^{i+1}$ without any stopovers,
		\item[(b)] or the distance between $s_p^i$ and $s_p^{i+1}$ is smaller than a given maximum walking distance ($\Delta_\text{w}$) and $t_p^{i+1} - t_p^i$ is at minimum the time necessary to walk from $s_p^i$ to $s_p^{i+1}$ with given walking speed $\xi$,
		\item[(c)] or the passenger waits at the stop, such that $s_p^i = s_p^{i+1}$ and $t_p^i > t_p^{i+1}$.
	\end{itemize}
 
\item[(iii)] The distance between $o_p$ and $s_p^1$ is smaller than a maximum access distance $\Delta_\text{a}$.
\item[(iv)] The distance between $s_p^{n_p-1}$ and $d_p$ is smaller than a maximum egress distance $\Delta_\text{e}$.
\item[(v)] The time $t_p^1 - \delta_p$ until a vehicle arrives at the first stop $s_p^1$ when a passenger first enters the transportation system is smaller than a given maximum waiting time $\Upsilon$.
\item[(vi)] $t_p^{n_p} \leq \delta_p + t_{\text{max}}$, where $t_{\text{max}}$ is the maximum travel time a passenger is allowed to take for its trip.
\end{itemize}

\noindent Three comments on this problem setting are in order. First, we assume that information about stops and arrival times at the stops is available for all vehicle routes. This assumption holds for scheduled transportation modes, e.g., bus, subway, and tram, but does not apply to ride hailing services. Still, it is possible to include unscheduled transportation modes in our model by adding a fully connected graph layer. Second, we assume fixed arc capacities with fixed travel times instead of flow dependent travel times. This approximation is mostly consistent with travel time observations on real-world roads. Here, travel time tends to stay stable up to the roads capacity and increases precipitously afterwards \citep{li2011fundamental}. An extensive discussion on the validity of this modeling assumption can be found in \cite{varaiya2005we}, and \cite{ostrovsky2019carpooling}. Third, we assume a fixed maximum travel time $t_{\text{max}}$ for all passengers. In practice, our approach is not limited to a homogeneous $t_{\text{max}}$ but allows to chose $t_{\text{max}}$ passenger-dependent. From an algorithmic perspective, a passenger-dependent $t_{\text{max}}$ reduces the size of the problem`s multilayered digraph. Henceforth, we decided to chose a homogeneous and large $t_{\text{max}}$ in order to challenge our algorithmic framework.

\section{Methodology}
\label{sec:methododlogy}
In the following, we introduce an algorithmic framework to solve the planning problem introduced in Section~\ref{sct: problem setting}. Intuitively, we can obtain a solution $\boldsymbol{\psi}$, i.e., paths for all passengers such that the sum over all travel times is minimal, by solving an integer minimum-cost MCNF problem. In a general minimum-cost MCNF problem let $\kappa_{ij}$ be the capacity of an arc $a=(i,j)$ with $i,j \in \mathcal{V}$, $d_i^k$ be the vertex demand of vertex $i\in \mathcal{V}$ and commodity $k\in \mathcal{K}$, and $x_{ij}^k \in \{0,1\}$ be the decision variable indicating whether commodity $k \in \mathcal{K}$ uses arc $(i,j)\in \mathcal{A}$ ($x_{ij}^k = 1$) or not ($x_{ij}^k = 0$). The vertex demand for commodity $k$ at vertex $i \in \mathcal{V}$ is defined as 

\begin{equation*}
  d_{i}^k =
    \begin{cases}
      \text{$1$}& \text{if } i = u_k,\\
      \text{$-1$}& \text{if } i = v_k,\\
      \text{$0$}& \text{otherwise}
    \end{cases}       
\end{equation*}

\noindent Here, $\mathcal{K}$ is a set of commodities, $u_k$ the source, and $v_k$ the sink of commodity $k \in \mathcal{K}$.

\begin{equation*}
\tag{IP 1}
\begin{array}{lll@{}lll}
    \min_{x}  & \displaystyle\sum\limits_{k \in \mathcal{K}} \displaystyle\sum\limits_{(i,j) \in \mathcal{A}} c_{ij} \ x_{ij}^k &&& (1a)\\
    \text{s.t.}& \displaystyle\sum\limits_{j \in \mathcal{N}^+(i)} x_{ij}^k - \displaystyle\sum\limits_{j \in \mathcal{N}^-(i)} x_{ji}^k = d_{i}^k,  && i \in \mathcal{V}, k \in \mathcal{K}&(1b)\\
    & \displaystyle\sum\limits_{k \in \mathcal{K}} x_{ij}^k \leq \kappa_{ij}, &&(i,j) \in \mathcal{A}&(1c)\\
    & x_{ij}^k \in \{0,1\}, &&(i,j) \in \mathcal{A}, k \in \mathcal{K} &(1d)
\end{array}
\label{eqn: IP 1}
\end{equation*}

\noindent In \ref{eqn: IP 1}, the objective function minimizes the sum over the cost of all commodity flows, while Constraints (1b) ensure flow conservation with $\mathcal{N}^+(i)$/$\mathcal{N}^-(i)$ being the outgoing/ingoing neighbourhood of $i \in \mathcal{V}$. Note that an ingoing neighbor of $i$ is a vertex $j\in \mathcal{V}$ such that $(j,i)\in \mathcal{A}$; an outgoing neighbor of $i$ is a vertex $j$ such that $(i,j)\in \mathcal{A}$. The remaining constraints enforce the capacity constraints of all arcs (1c) and ensure integer commodity flows~(1d).

To utilize such a minimum-cost MCNF formulation for our problem setting, it is necessary to model the transportation system as a digraph, which includes origin and destination vertices for all passengers, considers the intermodality of our problem setting, and encodes the spatial and temporal route constraints.

We introduce our methodology in three steps. First, we introduce the graph formulation that combines two core principles of network optimization: layered-graph structures and (partially) time-expanded networks. This allows us to implicitly encode problem specific constraints related to the intermodality of our planning problem and enables us to solve a standard integer MCNF problem to determine a solution $\boldsymbol{\psi}$. Second, we present a CG procedure to solve continuous MCNF problems. Third, we elaborate on how to obtain integral solutions through the application of a P\&B approach.

\subsection{Graph formulation}
\label{subsec: graph formulation idea}
To solve the planning problem as introduced in Section \ref{sct: problem setting} with a standard minimum-cost MCNF formulation, we capture network specific constraints implicitly in the underlying graph structure, independently from the MCNF formulation. To do so, we model the transportation system as a multilayered digraph as schematically shown in Figure~\ref{fig:Network graph}. This digraph contains a route layer for each route in the transportation system, which allows passengers to use a vehicle operating a route. Additionally, waiting layers for each stop in the transportation system allow passengers to wait at a stop for a new vehicle to arrive. Transit arcs enable passengers to switch between route and waiting layers and thus to enter or leave a vehicle at a stop. Walking arcs between different stop layers allow passengers to walk between stops in a certain proximity. 

\begin{figure}[!b]
	\fontsize{9}{10}\selectfont
    \centering
    \def\svgwidth{0.9\columnwidth}
    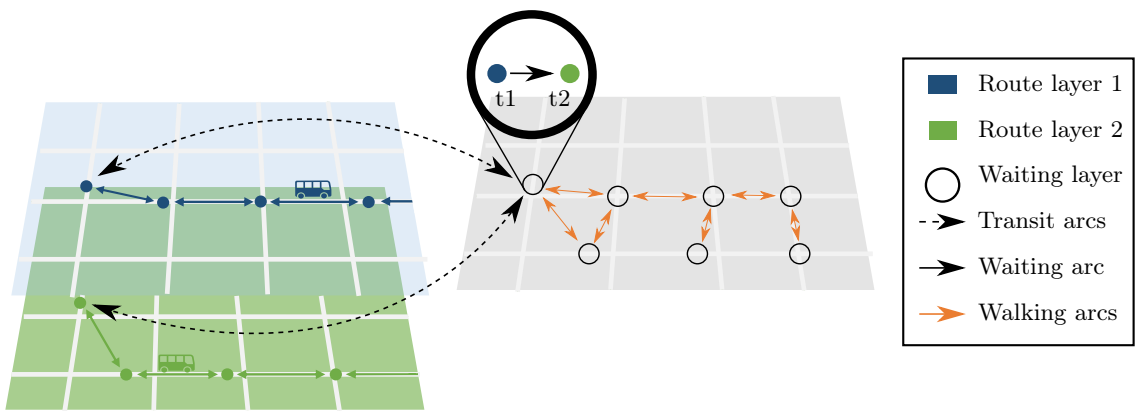
        \caption{Scematic illustration of the multilayered network structure}
    \label{fig:Network graph}
\end{figure}

\begin{figure}[!t]
\centering
\begin{subfigure}{.5\textwidth}
  \centering
  \includegraphics[width=.4\linewidth]{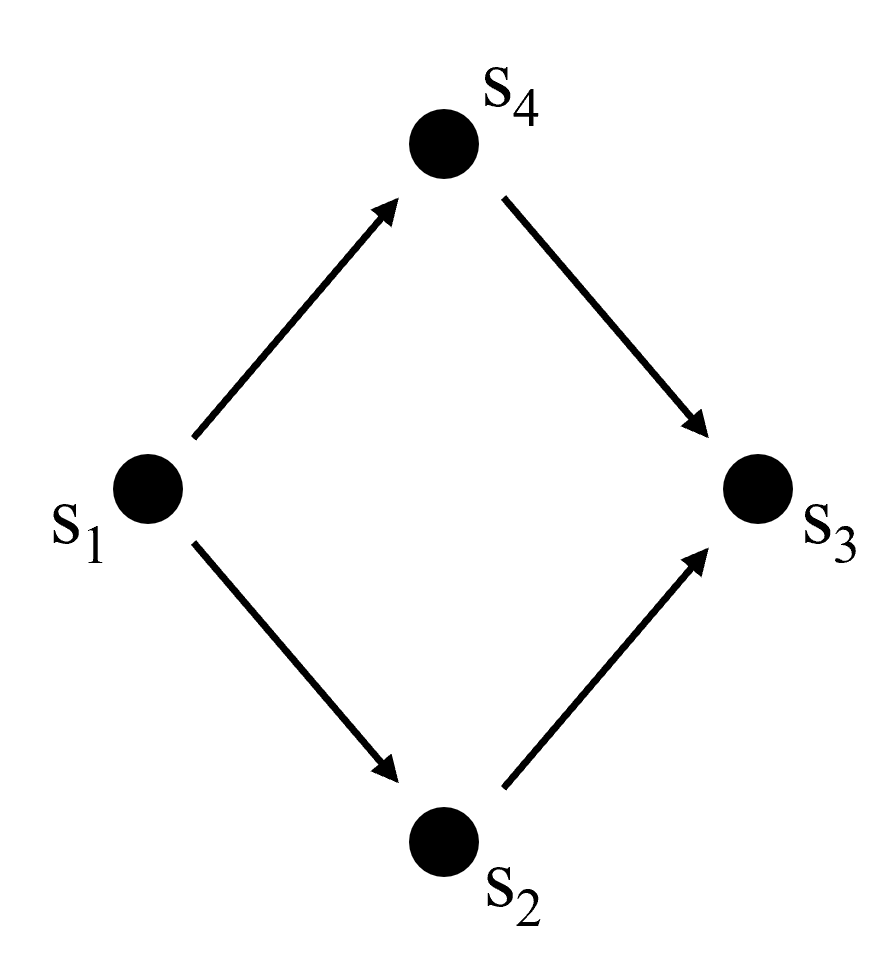}
  \caption{A static graph}
  \label{fig:static graph}
\end{subfigure}%
\begin{subfigure}{.5\textwidth}
  \centering
  \includegraphics[width=.6\linewidth]{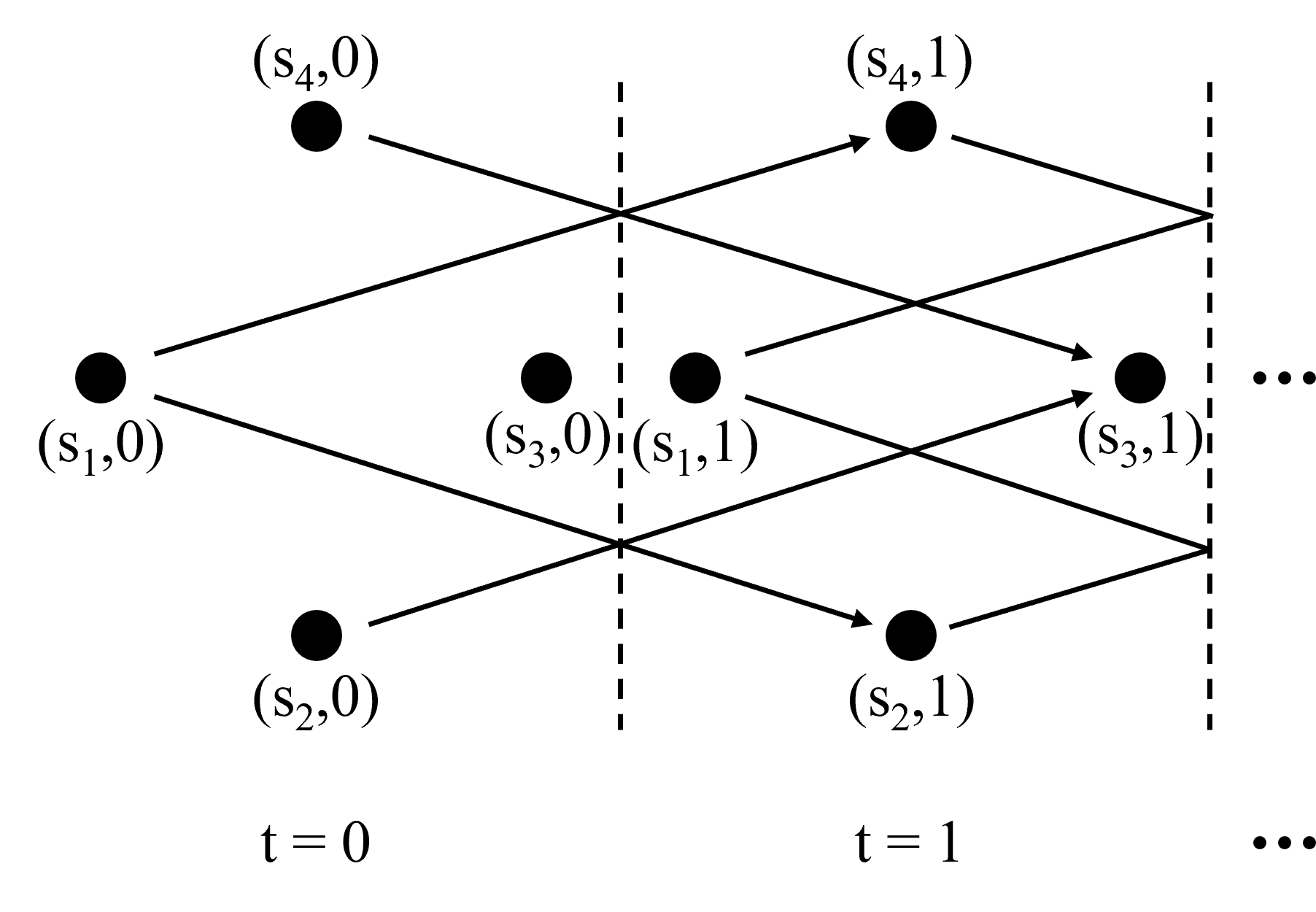}
  \caption{A time-expanded graph}
  \label{fig:time-expanded graph}
\end{subfigure}
\caption{Illustration of a static and a time-expanded graph with travel time one for all arcs}
\label{fig:static and time-expanded graph}
\end{figure}

Such a multilayered transportation network can be modeled as a static digraph $\mathcal{G}_{\sigma} = (\mathcal{V}_{\sigma}, \mathcal{A}_{\sigma})$ with a set of vertices $\mathcal{V}_{\sigma}$ and a set of arcs $\mathcal{A}_{\sigma}$. Here, the vertex set $\mathcal{V}_{\sigma}$ represents the stops $\mathcal{S}$ of the transportation system and an arc $a=(s_1,s_2) \in \mathcal{A}_{\sigma}$ indicates that a route $r \in \mathcal{R}$ exists, covering two consecutive stops $s_1$ and $s_2$. The drawback of such a static graph formulation is the missing time dimension such that the static graph formulation does not reflect the vehicle schedules.
Time-expanded graphs as introduced by \cite{ford1958constructing} allow for the integration of the time dimension (see Figure~\ref{fig:static and time-expanded graph}). A time-expanded graph $\mathcal{G}_{\Theta} = (\mathcal{V}_{\Theta},\mathcal{A}_{\Theta})$ contains one copy of the vertex set of an underlying static graph $\mathcal{V}_{\sigma}$ for each timestep. Additionally, arcs connect vertices in different time steps to account for the transit time.

While a time-expanded graph allows to implicitly model time, its size can grow rapidly even for relatively small underlying static graphs. To mitigate this drawback, we extend our layered graph with a partial time expansion \citep{Boland2017}. Here, the general idea is to copy only vertices at which a vehicle arrives at timestep $t$ and to only add a temporal arc $a = ((s_1, t_1)(s_2,t_2))$ to the time-expanded graph if a route covers stop $s_1$ in timestep $t_1$ and subsequently stop $s_2$ in timestep $t_2$. Such a graph formulation keeps all relevant vertices and arcs in its temporal expansion but the cardinality of the respective vertex and arc set is significantly smaller than for a fully time-expanded graph.

Accordingly, we represent an intermodal transportation system as a partially time-expanded multilayered digraph $\mathcal{G = (V, A)}$ with a set of temporal vertices $\mathcal{V}$ and a set of temporal arcs $\mathcal{A \subseteq V \times V}$. Let $\mathcal{S}$ be the set of all stops of the transportation network and let $\mathcal{T} = \{\bigcup_{p \in \mathcal{P}} \delta_p\} \cup \{\bigcup_{s \in \mathcal{S}} T(s)\} \cup \{\bigcup_{p \in \mathcal{P}} \delta_p + t_{\text{max}}\}$ be the set of all relevant timesteps. This set consists of all passenger departure times $\delta_p$, all timesteps during which any vehicle arrives at any station $T(s)$, and the latest timestep in which each passenger is allowed to arrive at its destination $\delta_p + t_{\text{max}}$. Then, the graph $\mathcal{G}$ contains a route layer $G_r = (V_r, A_r)$ for each vehicle route $r\in \mathcal{R}$ and a waiting layer $G_s = (V_s, A_s)$ for each stop $s\in \mathcal{S}$. In the remainder of this subsection, we detail the construction of $\mathcal{G}$. 

\subsubsection{Construction of route and waiting layers}
The route layers $G_\mathcal{R} = \bigcup_{r \in \mathcal{R}}G_r$ contain temporal route vertices $V_\mathcal{R} = \bigcup_{r \in \mathcal{R}}V_r$ and arcs $A_\mathcal{R} = \bigcup_{r \in \mathcal{R}}A_r$. 
A temporal route vertex $i = (s, t, r) \in V_r$ of a route layer $G_r$ with $r \in \mathcal{R}$ represents the vehicle of route $r$ arriving at stop $s\in \mathcal{S}$ at timestep $t\in \mathcal{T}$. Consecutive stops of route $r$ represented by temporal nodes $i,j \in V_r$ with $i = (s_i, t_i, r)$ and $j = (s_j, t_j, r)$ are connected via a temporal arc $a = (i,j) \in A_r$. Each arc in $A_r$ has a cost $c_{ij} = t_j - t_i$ and a capacity $\gamma_r$ which is equal to the capacity of the vehicle operating route $r$. 

The waiting layers $G_\mathcal{S} = \bigcup_{s \in \mathcal{S}}G_s$ contain temporal waiting vertices $V_\mathcal{S} = \bigcup_{s \in \mathcal{S}}V_s$ and arcs $A_\mathcal{S} = \bigcup_{s \in \mathcal{S}}A_s$ which allows passengers to wait at a stop. 
A temporal waiting vertex $i = (s, t) \in V_s$ is a tuple comprising a stop $s \in \mathcal{S}$ and a timestep $t \in \mathcal{T}$. For each timestep $t \in \mathcal{T}$ in which a vehicle of route $r\in \mathcal{R}$ arrives at stop $s\in \mathcal{S}$, there exists a temporal vertex $i = (s,t) \in V_s$ in the waiting layer of stop $s$. Consecutive temporal vertices are connected by waiting arcs $a \in A_s$. Waiting arcs have unlimited capacity, as we assume sufficient space at the stops for all passengers. The cost of waiting arcs results analogously to route arcs. 

\subsubsection{Construction of transit and walking arcs}
So far all route layers and waiting layers are disconnected. Transit arcs $A_T$ connect these different route and waiting layers and can either point from a route layer $G_r$, $r \in \mathcal{R}$ to a waiting layer $G_s$, $s \in \mathcal{S}$ or vice versa. Such a transit arc $a = (i, j) \in A_T$ from route node $i = (s, t, r) \in V_r$ to waiting node $j = (s, t) \in V_s$ has infinite capacity and cost $c_{ij} = 0$. For every such transit arc $a = (i,j)$ pointing from a route layer to a waiting layer, there exists a corresponding transit arc $a' = (j,i)$ from the waiting layer to the route layer with identical capacity and cost. 
Walking arcs $A_W$ connect waiting nodes of stops within a predefined walking distance $\Delta_\text{w}$. Here, waiting nodes $i = (s, t_i)$ and $j = (s', t_{j})$ of different stops $s, s' \in \mathcal{S}$ with $dist(s,s') \leq \Delta_\text{w}$ are connected with a walking arc $a = (i,j)$ with cost $c_{ij} = t_j - t_i$ and unconstrained capacity, if $t_i + \frac{dist(s,s')}{\xi} \leq t_j$. Here, $dist(s, s')$ denotes a distance measure between the stop coordinates of two stops $s,s' \in \mathcal{S}$.\\
The resulting graph $\mathcal{G = (V, A)}$ with $\mathcal{V} = V_\mathcal{R} \cup V_\mathcal{S}$ and $\mathcal{A} = A_\mathcal{R} \cup A_\mathcal{S}\cup A_T\cup A_W$ represents the temporal and spatial expansion of the transportation network.

\subsubsection{Construction of access and egress arcs}
To formulate the underlying optimization problem as an MCNF problem, we add the origin vertex $\omega_p^o = (o_p, \delta_p)$ and the destination vertex $\omega_p^d = (d_p, \delta_p + t_{\text{max}})$ of each passenger $p \in \mathcal{P}$ to $\mathcal{G}$. We can now connect $\omega_p^o$ to all stops $s \in \mathcal{S}$ with $dist(o_p, s') \leq \Delta_\text{a}$, for which a waiting node $v = (s, t)$ with $\delta_p + \frac{dist(o_p, s)}{\xi} \leq t \leq \delta_p + \Upsilon$ exists. These access arcs $a = (\omega_p^o, v)$ have cost $c_{\omega_p^o,v} = t - \delta_p$. Egress arcs connecting the transportation network with the destination vertex $\omega_p^d$ originate from all waiting nodes $v = (s, t) \in V_s$ of a stop $s\in \mathcal{S}$ with $dist(s, d_p) \leq \Delta_\text{e}$ for which it holds that $\delta_p \leq t + \frac{dist(s, d_p)}{\xi} \leq \delta_p + t_{\text{max}}$. These egress arcs have cost $c_{v,\omega_p^d} = \frac{dist(s, d_p)}{\xi}$. 

Thus, we expand the digraph $\mathcal{G = (V, A)}$ by the set of origin vertices $V_O$ and the set of destination vertices $V_D$ as well as the access arcset $A_{O}$ and the egress arcset $A_{D}$ resulting in $\mathcal{V} =  V_\mathcal{R} \cup V_\mathcal{S} \cup V_O \cup V_D$ and $\mathcal{A} = A_\mathcal{R} \cup A_\mathcal{S}\cup A_T\cup A_W \cup A_{O} \cup A_{D}$. Each passenger travel demand is now represented as an individual flow going from an origin vertex to a destination vertex with flow demand one. This allows us to model our problem as a minimum-cost MCNF problem as introduced in \ref{eqn: IP 1}, where $\mathcal{K} = \mathcal{P}$ and the origin and destination vertices correspond to sources and sinks for each commodity flow.

\subsubsection{Illustrative example}
\noindent \label{subsec: example}
\noindent The following example illustrates the resulting graph $\mathcal{G}$ for a transportation network. Table \ref{tab: vehicle routes} depicts the timetables for three vehicle routes $\mathcal{R} = \{r_1,r_2,r_3\}$, which operate on a set of stops $\mathcal{S} = \{s_1, s_2, s_3\}$. All vehicle routes $r \in \mathcal{R}$ have capacity $\gamma_r = 1$. We only consider one passenger with request tuple $\zeta_p = (o_p, d_p, 0)$. Table \ref{tab: walking distance} depicts the access and egress distances, as well as the walking distances between the different stops. 

\begin{table}[bp]
\small
\caption{Network characteristics}
\centering
\begin{subtable}{.5\linewidth}
\centering
\caption{Arrival times for vehicle routes}
\begin{tabular}{cccc}
& \multicolumn{1}{l}{\textbf{$r_1$}} & \multicolumn{1}{l}{\textbf{$r_2$}} & \multicolumn{1}{l}{\textbf{$r_3$}} \\ \hline
$s_1$ & 5 & - & 1 \\
$s_2$ & 6 & 2 & 2 \\
$s_3$ & - & 3 & 5 \\
\end{tabular}
\label{tab: vehicle routes}
\end{subtable}%
\begin{subtable}{0.5\linewidth}
	\small
	\centering
	\caption{Distances in the transportation network}
	\label{tab: walking distance}
	\begin{tabular}{cccccc}
		\multicolumn{1}{l}{} & \multicolumn{1}{l}{$s_1$} & \multicolumn{1}{l}{$s_2$} & \multicolumn{1}{l}{$s_3$} & \multicolumn{1}{l}{$o_1$} & \multicolumn{1}{l}{$o_2$}\\ \hline
			$s_1$                           & 0          & 6 		& 1		& 3		& 10 \\
			$s_2$                           & 6          & 0  	& 4 		& 5		& 1	\\
			$s_3$                           & 1          & 4  	& 0 		& 2		& 7 \\	
	\end{tabular}
	\end{subtable}%
\end{table}

We assume a walking speed $\xi = 1$ for access, egress and walking distances. Furthermore, we set the maximum access distance $\Delta_\text{a} = 3$, the maximum egress distance $\Delta_\text{e} = 7$, and the maximum walking distance $\Delta_\text{w} = 4$. Additionally, we set the maximum waiting time $\Upsilon = 4$ and the maximum travel time $t_{\text{max}}=10$. Figure~\ref{fig: example} shows the multilayered partially time-expanded graph for this example.

\begin{figure}[!b]
    \centering
    \def\svgwidth{\columnwidth}
    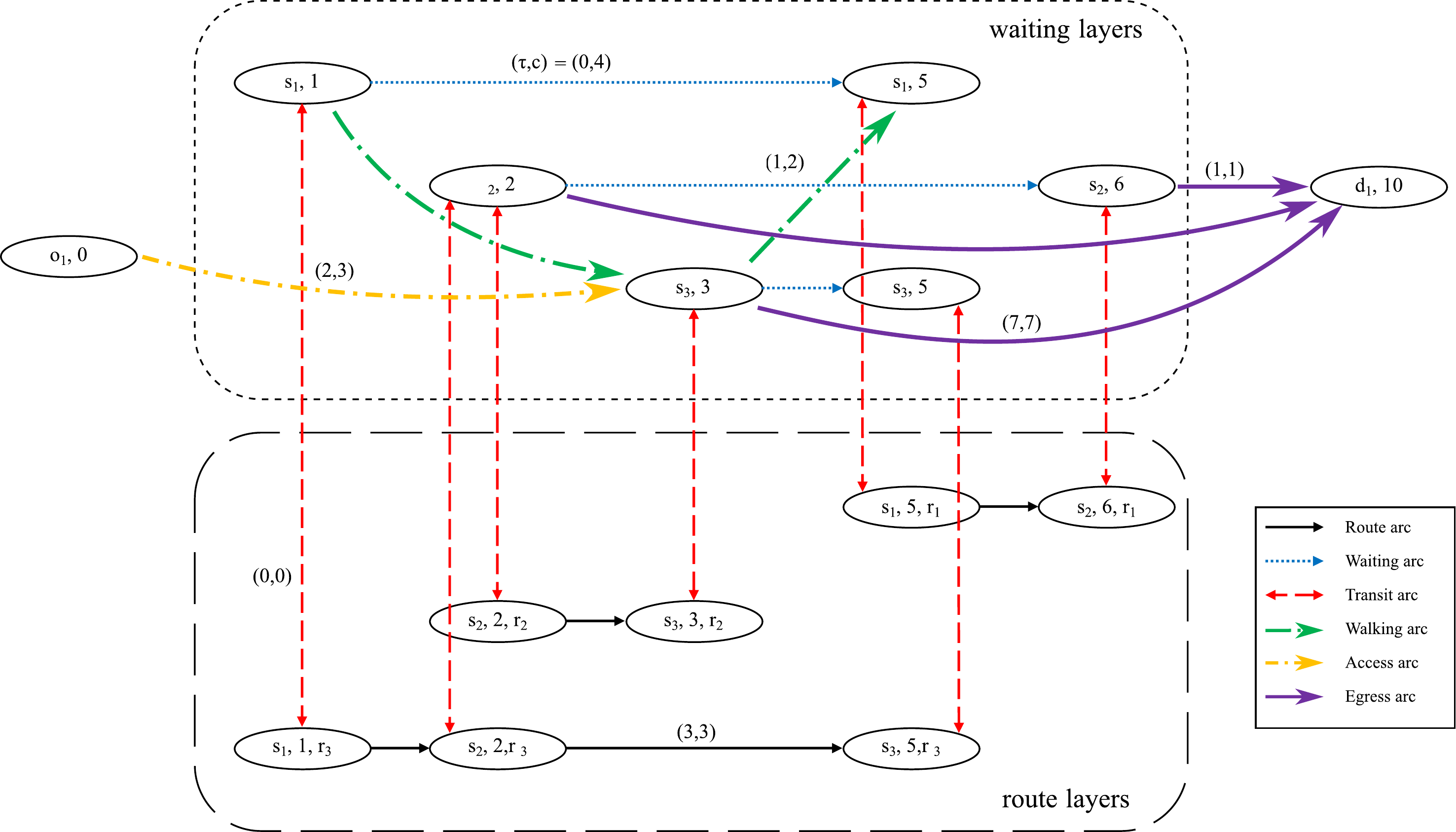
        \caption{Partially time-expanded multilayered digraph of the example with travel time $\tau$ and arc \\ cost $c$}
    \label{fig: example}
\end{figure}

\noindent For each route arc the travel time $\tau$ and the arc cost $c$ coincide, see, e.g., arc (($s_2,2,r_3$),($s_3,5,r_3$)). Nodes from the route layers are connected to nodes from the waiting layers via transit arcs. These transit arcs have a travel time and cost of zero, see, e.g., arc (($s_1,1,r_3$),($s_1,1$)). Nodes in the same waiting layer are connected via waiting arcs, all having a travel time of zero, but an arc cost equivalent to the waiting time, see, e.g., arc (($s_1,1$),($s_1,5$)). The origin vertex ($o_1,0$) is only connected to node ($s_3,3$) with travel time two and arc cost three. Note that even though $s_1$ has an access distance of three, there is no access arc to this waiting layer because the first possible departure after the passenger arrives at $s_1$ is at timestep five which exceeds $\Upsilon = 4$. There are two walking arcs between the stop layer of $s_1$ and $s_3$. Both walking arcs have a transit time of one and an arc cost of two. Note that even though the walking distance between $s_2$ and $s_3$ is smaller than $\Delta_\text{w}$, there exist no walking arcs between these two waiting layers because the travel time is too long, e.g., a passenger which starts at ($s_2,2$) can earliest arrive at $s_3$ at timestep six but the latest possible timestep at which to arrive at the waiting layer of $s_3$ is five. The destination vertex $(d_p, 10)$ has three ingoing egress arcs. For each egress arc the travel time and arc cost coincide. There exists no egress arc from ($s_3,5$) because a passenger cannot arrive at the destination before timestep 12, which is greater than $\delta_p + t_{\text{max}} = 10$. Thus, the two feasible paths for the passenger
are
\begin{align*}
&\eta_1 = ((o_1,0),(s_3,3),(d_1,10))\\
&\eta_2 = ((o_1,0),(s_3,3),(s_1,5),(s_1,5,r_1),(s_2,6,r_1),(s_2,6),(d_1,10))
\end{align*}

\noindent Here, the cost of $\eta_1$ is ten and the cost of $\eta_2$ is seven. Thus, $\eta_2$ can be used to construct path $\phi_1 = [(o_1,0),(s_3,3),(s_1,5),(s_2,6),(d_1,7)]$ for our problem setting. Note that the elements of $\phi_1$ do not correspond to nodes of Figure~\ref{fig: example}.

\subsection{Column generation}
\label{subsection: column generation}
For large sets of passengers, Constraints (1b) in \ref{eqn: IP 1} grow exponentially, which makes this standard model formulation of minimum-cost MCNF problems non-tractable for large instances. To mitigate this issue, we reformulate \ref{eqn: IP 1} as a set-covering problem to solve it via CG, which can be applied to linear programs (LPs) with a large set of decision variables. Here, the idea is to solve the LP with only a subset of all variables and iteratively add new variables to the program which have the potential to improve the objective function. New promising variables can be found by solving so called pricing problems. If one can show that adding new variables cannot improve the objective value, the algorithm terminates.

\begin{equation}
\tag{IP~2}
    \begin{array}{llll@{}llll}
        \min_{\bm{\lambda}} &\bm{c}^T\bm{Y\lambda} & (2a)\\
        \text{s.t.} &\bm{Y\lambda} \leq \bm{\kappa} & (2b) \\
        & \bm{\Lambda\lambda} = \bm{1}^{\vert \mathcal{P}\vert} &(2c)\\
        & \bm{\lambda} \in \{0,1\}^{n} & (2d)
    \end{array}
    \label{eqn: IP 2}
\end{equation}

\noindent \ref{eqn: IP 2} denotes the set covering reformulation of \ref{eqn: IP 1}, in which all possible paths for flows are represented as columns in the matrix $\bm{Y} \in \mathbb{R}^{\vert \mathcal{A} \vert \times n}$, where $n$ is the number of all possible paths and $\bm{Y_v}$ is the \textit{v}\textsuperscript{th} column of matrix $\bm{Y}$. The cost vector $\bm{c} \in \bm{\mathbb{R}}^{|\mathcal{A}|}$ is a vector containing the cost of all arcs $a \in \mathcal{A}$. The vector $\bm{\kappa} \in \bm{\mathbb{R}}^{\vert \mathcal{A} \vert}$ is the capacity vector of all arcs $a\in \mathcal{A}$ and the matrix $\bm{\Lambda}\in \bm{\mathbb{R}}^{\vert \mathcal{P}\vert \times n}$ is the incidence matrix between the paths of $\bm{Y}$ and the flows of passengers $p\in \mathcal{P}$. Accordingly, $\bm{\Lambda}_{pv} = 1$, if path $\bm{Y_v}$ corresponds to the flow of passenger $p$ and $\bm{\Lambda}_{pv}= 0$ otherwise. The decision variable vector $\bm{\lambda} \in \{0,1\}^n$ selects exactly one feasible path for each flow.
\vspace{\baselineskip}

\noindent To apply CG, we relax the integer constraint (2d) to a non-negativity constraint and solve the continuous relaxation of \ref{eqn: IP 2}. Note that the solution of the continuous relaxation is a lower bound of the integer solution and thus the system optimum. We add dummy variables with objective costs greater than $t_{\text{max}}$ for each passenger. This allows us to initiate the CG algorithm with an empty subset of all decision variables, i.e., without knowing any columns of matrix $Y$. Then, the resulting restricted master problem ($RMP$) based on \ref{eqn: IP 2} is as follows

\begin{equation}
\tag{RMP}
    \begin{array}{llll@{}llll}
        \min_{\bm{\lambda}} &\bm{c}^T\bm{Y\lambda} + \bm{c_{\text{init}}}^T\bm{\lambda}& (3a)\\
        \text{s.t.} &\bm{Y\lambda} \leq \bm{\kappa} & (3b) \\
        & \bm{\Lambda\lambda} = \bm{1}^{\vert \mathcal{P}\vert} &(3c)\\
        & 0 \leq \bm{\lambda}[i], i \in [1,...,|\mathcal{P}|]& (3d)
    \end{array}
    \label{eqn: RMP}
\end{equation} 

\noindent Here, $\bm{\Lambda} = \mathbb{I}_{|\mathcal{P}| \times |\mathcal{P}|}$, where matrix $\mathbb{I}_{|\mathcal{P}|\times |\mathcal{P}|}$ is the identity matrix of dimension $|\mathcal{P}|$. Additionally, matrix $\bm{Y} = \bm{0}_{|\mathcal{P}| \times |\mathcal{P}|}$, where $\bm{0}_{|\mathcal{P}| \times |\mathcal{P}|}$ is the matrix of all zeros. Furthermore, the expression $\bm{c_{\text{init}}}^T \bm{\lambda}$ is added to the objective function of \ref{eqn: IP 2}. Here, $\bm{c_{\text{init}}}[i] = \rho$ for $i \in [1,...,|\mathcal{P}|]$ is the penalty which occurs if a passenger is not able to reach its destination within $t_{\text{max}}$ timesteps. We then iteratively add new columns to the linear program by solving \ref{eqn: LP 4}, i.e., a pricing problem, for each passenger $p \in \mathcal{P}$.

\begin{equation*}
\tag{LP 4}
\begin{array}{llll@{}llll}
    \min_{\bm{X^p}} &(\bm{c} - \bm{w^*})^T \bm{X^p} - \alpha_p^* & (4a)\\
    \text{s.t.} &\bm{X^p} \leq \bm{\kappa} & (4b)\\
    & \bm{BX^p} = \bm{d^p} & (4c)\\
    & \bm{X^p} \geq \bm{0} & (4d)
\end{array}
\label{eqn: LP 4}
\end{equation*}

\noindent Here, $\mathbf{c}$ is the cost vector representing the costs of the arcs in $\mathcal{G = (V,A)}$. The vector $\bm{w^*}$ is the vector of the dual variables for the capacity constraints (3b), $\bm{\alpha^*}$ is the vector of the dual variables for the convexity constraint (3c), and $\bm{X^p}$ is the new path we want to find for the flow of passenger $p \in \mathcal{P}$. The matrix $\bm{B} \in \mathbb{R}^{\vert \mathcal{V} \vert \times \vert \mathcal{A} \vert}$ is the vertex-edge incidence matrix of graph $\mathcal{G}$. If the objective value for the pricing problem of passenger $p$ is negative, we add $\bm{X^p}$ as a new column to $\bm{Y}$ in \ref{eqn: RMP} and expand $\bm{\lambda}$ to account for the new found path. Notice that the constraint $\bm{\lambda}[i] \leq 1, i \in [1,...,|\mathcal{P}|]$ is redundant in the \ref{eqn: RMP} because of constraint (3c).

 \begin{algorithm*}[!t]
 \footnotesize
 \caption{Column generation pseudocode}\label{alg:cg}
 \begin{algorithmic}[1]
 \State $RMP \gets$ build \ref{eqn: RMP} based on \ref{eqn: IP 2}
 \State LB $\gets$ 0
 \While{$TRUE$}
     \State $sol \gets$ solve($RMP$)
     \If{$RMP$ infeasible}
         \State RETURN \textit{Error: RMP infeasible}
     \EndIf
    
    \If{Dual variables of $RMP$ did not change or $LB = sol.objVal$}
    		\State RETURN sol
    	\EndIf
%     \State $w^* \gets$ (2$b$) dual variables
%     \State $\alpha^* \gets$ (2$c$) dual variables

%     \State $PricingSet \gets$ Filter($w^*, \alpha^*$)
     \State $\Pi \gets$ $\mathcal{P}$
    
%     \If{$PricingPool = \emptyset$}
%         \State RETURN $sol$
%     \EndIf
	 \State $\beta$ = $RMP.objVal$
     \For{$p \in \Pi$}
         \State $pp \gets solvePricingProblem(p)$
         \If{$pp.objVal$ $<$ 0}
             \State addNewColumn($RMP$)
             \State $\beta$ += pp.ObjVal
         \EndIf
         \EndFor
     \If{$\beta$ $>$ LB}
     	\State LB $\gets$ $\beta$
     \EndIf
%     \If{NewColumns == 0}
%     	\State RETURN sol
%     \EndIf
 \EndWhile
 \end{algorithmic}
 \label{ALG: 1}
 \end{algorithm*}

\noindent Algorithm \ref{ALG: 1} shows the pseudocode of our CG approach. We start by initializing the \ref{eqn: RMP}~(l.~1) and setting the lower bound~(LB) of the continuously relaxed master problem to zero~(l.~2). In every iteration, we solve the current $RMP$~(l.~4) and ensure its feasibility~(l.~5). Afterwards two stopping criteria are checked. The first stopping criterion checks if the dual variables of the \ref{eqn: RMP} did not change after it was solved~(l.~8). If the dual variables of the $RMP$ did not change, no new columns can be found. The second stopping criterion checks if the lower bound $LB$ is equal to the current objective function $sol.objVal$ of the $RMP$. If one of the two stopping criteria is fulfilled, we found an optimal solution for the continuous master problem and return it~(l.~9). If the algorithm did not terminate in the current iteration, we define a pricing pool~($\Pi$) which in its basic variant is equal to the set of all passengers $\mathcal{P}$~(l.~11). After that, we set a new variable $\beta$ to the current objective value of the $RMP$ which is used to try to find a new lower bound for the continuously relaxed master problem~(l.~12). Afterwards, the pricing problems of all passengers in $\Pi$ are solved~(l.~14). If the objective value of the pricing problem is negative, a new column is added to the $RMP$~(l.~16). Here, we add a new column to the $RMP$ by:
 \begin{itemize}
 	\item[(i)] Adding $\bm{X^p}$ as a new column to $\bm{Y}$.
 	\item[(ii)] Appending 0 to the cost vector $\bm{c_{\text{init}}}$. Note that the cost vector $\bm{c}$ does not change.
 	\item[(iii)] Appending the column vector $\bm{l} \in \mathbb{R}^{|\mathcal{P}|}$ to $\bm{\Lambda}$ where $\bm{l}[p] = 1$ and $\bm{l}[i] = 0$ for $i \in [1,...,|\mathcal{P}|]\backslash p$. 
 	\item[(iv)] Incrementing the dimension of $\bm{\lambda}$ by one.
 \end{itemize}
After all new columns are added, we update the $RMP$'s lower bound (l. 20-22). After the algorithm solved all pricing problems and updated the lower bound, we start a new iteration of the CG (l. 4).

While this CG implementation is straightforward, its computing time can be improved by utilizing a pricing filter to reduce the size of $\Pi$ (l.~11), which leads to fewer pricing problems that need to be solved (l.~13-19), and by using an $A^*$-based pricing algorithm instead of a straightforward Dijkstra-based pricing algorithm \citep{dijkstra1959note} to solve the pricing problems. In the remainder of this section, we elaborate on both of these improvement levers.

\subsubsection{Pricing Filter}
Solving the pricing problem $|\Pi|$ times in every iteration of the CG is a key bottleneck of Algorithm \ref{ALG: 1}. To mitigate this bottleneck, we reduce the pricing pool size $|\Pi|$ by adding a pricing filter. To derive such a pricing filter, we analyse the dual variables of the capacity constraints (3b). Intuitively, a negative dual capacity variable $\bm{w^*}_{ij}$ indicates that more passengers try to use arc $a = (i,j)$ than the arc capacity $\bm{\kappa}_{ij}$ allows. Accordingly, we add all passengers with a possible path which uses an arc with a negative dual capacity variable to the pricing pool. Note that in this case, the decision variable of the determining path does not have to be in the basis of the LP solution.

\begin{algorithm*}[!b]
 \footnotesize
 \caption{Pricing filter algorithm}
 \begin{algorithmic}[1]
 \State $\Pi \gets$ $\emptyset$
 
 \State $\Phi$ $\gets$ $\emptyset$
 \For{$(i,j) \in \mathcal{A}$}
 	\If{$w^*_{ij} < 0$}
 		\State $\Phi$.add($(i,j)$)
 	\EndIf
 \EndFor
 
 \For{$Y_v \in Y$}
 	\State usedArcs $\gets \emptyset$
 	\If{$Y_v((i,j)) > 0$}
 		\State usedArcs.add($(i,j)$)
 	\EndIf
 	\If{usedArcs $\cap$ $\Phi$ $\neq \emptyset$}
 		\State $\Pi$.add(corresponding passenger $p$ of $Y_v$)
 	\EndIf 
 \EndFor
 \end{algorithmic}
 \label{ALG: 2}
 \end{algorithm*}

Algorithm \ref{ALG: 2} shows the pseudocode of our pricing filter. We iterate through all arcs $a \in \mathcal{A}$ of graph $\mathcal{G = (V,A)}$ (l. 3)  and add arcs with a negative dual capacity constraint to the set $\Phi$ (l. 5). Next, we iterate through all columns $\bm{Y_v}$ of matrix $\bm{Y}$, i.e., the paths that we found so far and collect all arcs that the paths use (l. 9-12). If a path uses an arc with a negative dual capacity value, we add the corresponding passenger to $\Pi$ (l. 13-15). Note that when applying our pricing filter, we have to run the pricing problems for all passengers again, as soon as the objective value of the \ref{eqn: RMP} does not improve anymore, to ensure optimality.

\subsubsection{Pricing Problem}
Recall that we defined our MCNF problem such that each request has a flow value of one and each arc has a capacity of at least one. Accordingly, the capacity constraint (4b) is always fulfilled and we can solve the pricing problem \ref{eqn: LP 4} as a shortest path problem, in which we aim to find the shortest path from the origin to the destination node for every passenger $p \in \mathcal{P}$, with a new arc cost $c'_{ij} = c_{ij} - w^*_{ij}$ for all arcs $(i,j) \in \mathcal{A}$. In this setting, a negative objective value of the pricing problem \ref{eqn: LP 4} corresponds to the shortest path from $(o_p, \delta_p)$ to $(d_p,\delta_p + t_{\text{max}})$ in graph $\mathcal{G = (V,A)}$ with arc costs $\bm{c'}$ being shorter than $a^*_p$. Generally, constraint (3b) and the constraints (3d) ensure that this approach can also be used when not all flow values are fixed to one.

A straightforward approach to solve the pricing problem (\ref{eqn: LP 4}) as a shortest path problem is to use the Dijkstra algorithm. However, applying a straightforward Dijkstra finds its limits for large scale instances due to the size of $\mathcal{G}$. Here, we note that we cannot increase the Dijkstra's algorithm's efficiency by topological ordering as $\mathcal{G}$ contains cycles induced by the transit arcs. Accordingly, we utilize the $A^*$ algorithm to increase the efficiency of the shortest path calculation.

To apply the $A^*$ algorithm, we need to define an admissible distance approximation $h$, which the $A^*$ algorithm uses to guide its search. This distance approximation $h$ does not overestimate the cost of the shortest path from all nodes $i \in \mathcal{G}$ to all destination vertices $j \in V_D$ of all passengers in graph $\mathcal{G = (V,A)}$ with cost function $c'_{ij} = c_{ij} - w^*_{ij}$ for all arcs $(i,j) \in \mathcal{A}$. The more accurate the distance approximation is, the fewer nodes $A^*$ has to expand to find the shortest path.\medskip

\noindent In the following, we reduce our time-expanded multilayered digraph to a significantly smaller static digraph which allows us to introduce an admissible distance approximation. Let $\mathcal{G} = (\mathcal{V, A})$ be the graph defined in Section \ref{subsec: graph formulation idea} and let $\mathcal{H} = (\mathcal{V}_H, \mathcal{A}_H)$ be a new graph. Let $\mathcal{V}_H = \mathcal{S} \cup \mathcal{V}_D$ be the union of the set of all stops $\mathcal{S}$ of the transportation system and all destination vertices $\mathcal{V}_D$. Furthermore, for $i,j \in \mathcal{V}_H$ let $A_{ij}$ be the set of all arcs $(a,b) \in \mathcal{A}$ of one of the following types
\begin{itemize}
	\item[(i)] $a,b \in \mathcal{V}_{\mathcal{R}}$, with $a = (i,t_a,r_a)$, $b = (j, t_b, s_b)$
	\item[(ii)] $a \in \mathcal{V}_{\mathcal{R}}$, $b \in \mathcal{V}_{\mathcal{S}}$, with $a = (i,t_a,r_a)$, $b = (j, t_b)$
	\item[(iii)] $a \in \mathcal{V}_{\mathcal{S}}$, $b \in \mathcal{V}_{\mathcal{R}}$, with $a = (i,t_a)$, $b = (j, t_b, s_b)$
	\item[(iv)] $a \in \mathcal{V}_{\mathcal{S}}$, $b \in \mathcal{V}_{D}$, with $a = (i,t_a)$, $b = (j, t_b)$
\end{itemize}
If $|A_{ij}| > 0$, we add arc $(i,j)$ to $\mathcal{A}_H$. The cost of arc $(i,j)$ is  defined as $c_{ij}^\mathcal{H} = \min\{c_{ab}|(a,b) \in \mathcal{A}_{ij}\}$. 

The following first two definitions, introduce the stop of a temporal route or waiting vertex and the pricing cost of the arcs of our time-expanded digraph. This allows us to propose a distance approximation in the third definition.

\begin{definition}
Let $\mathcal{G = (V, A)}$ be the graph defined in Section \ref{subsec: graph formulation idea} and let $a \in \mathcal{V_R}\cup \mathcal{V_S}$ be a vertex. We define the stop of $a$ as 

\begin{equation}
\eta(a)=
  \begin{cases}
    s_a, & \text{if } a = (s_a, t_a, r_a)\in \mathcal{V}_\mathcal{R}\\
	s_a, & \text{if } a = (s_a, t_a)\in \mathcal{V}_\mathcal{S}
  \end{cases}
\end{equation}
\end{definition}

\begin{definition}
Let $\mathcal{G = (V, A)}$ be the graph defined in Section \ref{subsec: graph formulation idea} and let $i,j \in \mathcal{V}$ be two vertices. We define the pricing cost $c'$ of graph $\mathcal{G}$ as $c'_{ij} = c_{ij} - w^*_{ij}$, $(i,j) \in \mathcal{A}$ with $w^*_{ij}$ as defined in \ref{eqn: LP 4}.
\end{definition}

\begin{definition}
Let $h_p(i)$ be the cost of the shortest path from $i$ to $d_p$ in $\mathcal{H}$ with $i\in \mathcal{V}_H$ and $d_p$ the destination vertex of passenger $p\in\mathcal{P}$. We then define $h'_p(a)$ as follows

\begin{equation}
h'_p(a)=
  \begin{cases}
    h_p(\eta(a)), & \text{if } a\in \mathcal{V}_\mathcal{R}\cup \mathcal{V}_\mathcal{S}\\
	h_p(a), & \text{if } a\in \mathcal{V}_D
  \end{cases}
\end{equation}

\end{definition}

\begin{theorem}
$h'_p(a)$ is an admissible distance approximation for the cost of the shortest path from $a\in \mathcal{V} \backslash \mathcal{V}_O$ to the destination vertex $d_p$ of passenger $p\in \mathcal{P}$ in graph $\mathcal{G}$ with pricing cost $c'$.
\end{theorem}

\begin{proof}
For $h'_p(a)$ to be an admissible distance approximation for the pricing cost $c'$ of the shortest path from $a\in \mathcal{V} \backslash \mathcal{V}_O$ to the destination vertex $d_p$, $p\in \mathcal{P}$ in $\mathcal{G}$, we have to show that there cannot exist a path $\sigma = [x_1, x_2,...,x_{n-1},x_n]$ with $x_1 = a$ and $x_n = d_p$ in $\mathcal{G}$, for which it holds that 

\begin{equation}
	\sum_{i = 1}^{n-1} c'_{x_ix_{i+1}} < h'_p(a)
\end{equation}

\textit{Case 1:} 
Let $a\in \mathcal{V}_D$. If $a = d_p$ then it holds that $h'_p(a) = h_p(a) = h_p(d_p) = 0$. If $a \neq d_p$  then there exist no outgoing arcs from $a$ in $\mathcal{G = (V,A)}$ and thus no outgoing arcs from $a$ in $\mathcal{H}$. Accordingly, there exists no path from $a$ to $d_p$ in $\mathcal{G}$, which implies that the pricing cost of the shortest path is $\infty$. By construction of $\mathcal{H}$ there also does not exist a path from $a$ to $d_p$ in $\mathcal{H}$ which implies $h'_p(a) = \infty$.
\vspace{1ex}

\textit{Case 2:} Let $a \in \mathcal{V}_\mathcal{R} \cup \mathcal{V}_\mathcal{S}$. Furthermore, let $\sigma = [x_1, x_2,...,x_{n-1},x_n]$ with $x_1 = a$ and $x_n = d_p$ be the shortest path from $a$ to $d_p$ in $\mathcal{G}$ with pricing cost $c'$. By construction of $\mathcal{H}$ there (i) exists a path $\sigma = [\eta(x_1), \eta(x_2),...,\eta(x_{n-1}),\eta(x_n)]$ in $\mathcal{H}$. Because (ii) $c_{\eta(x_i)\eta(x_{i+1})}^\mathcal{H} = \min\{c_{e,f}|(e,f) \in \mathcal{A}_{\eta(x_i)\eta(x_{i+1}}\}$ for all $i \in [1,...,n]$ and (iii) $w^*_{ef} \leq 0$ for all $(e,f) \in \mathcal{A}$, it holds that $c'_{x_ix_{i+1}} \overset{(iii)}{\geq} c_{x_ix_{i+1}} \overset{(ii)}{\geq} c^\mathcal{H}_{\eta(x_i)\eta(x_{i+1})}$ for all $i \in [1,2,...,n]$. Thus, 
\begin{equation}
\sum_{i = 1}^{n-1} c'_{x_ix_{i+1}} \overset{(iii)}{\geq} \sum_{i = 1}^{n-1} c_{x_ix_{i+1}} \overset{(ii)}{\geq} \sum_{i = 1}^{n-1} c^\mathcal{H}_{\eta(x_i)\eta(x_{i+1})} \overset{(i)}{\geq} h'_p(a)
\end{equation}
$$\eqno \square$$

\end{proof}

\begin{remark}
Vertices $a \in \mathcal{V}_O$ have no ingoing arcs, which implies that the $A^*$ algorithm cannot expand such a node except as the origin of a shortest path. Accordingly, we do not need a distance approximation for $a \in \mathcal{V}_O$ to apply the $A^*$ algorithm in the pricing problem of every passenger. Thus, we do not consider this case in the distance approximation.
\end{remark}

Since distance approximation $h'_p$ is admissible, optimality of the shortest path solution holds when we apply the $A^*$ algorithm to the pricing problem of passenger $p\in \mathcal{P}$ with distance approximation $h'_p$.

A straightforward approach to calculate the distance approximation is to initiate the Dijkstra algorithm from every vertex in $\mathcal{H}$, with all destination nodes in the target set. 
We can achieve a speedup in the computation time of the distance approximation if we preprocess a major part of the distance approximation calculations. To do so, let $A_Q = \{(q,j) \vert j \in V_D, q \in \mathcal{N}^-(j)\}$, where $\mathcal{N}^-(j)$ is the ingoing neighbourhood of $j \in V_D$ in graph $\mathcal{H}$ and $\mathcal{H'} = (V_H \setminus V_D, A_H \setminus A_Q)$ is a new graph, which does not contain the destination vertices. We can then calculate the shortest path $\sigma(i,q)$ between all vertex pairs in $\mathcal{H'}$ via the Dijkstra algorithm. This calculation of the all pairs shortest path problem in $\mathcal{H'}$ is equal for all instances on the same underlying transportation network and can thus be preprocessed. When we solve a new instance, we only have to find a path of minimum length $min\{\sigma(i,q) + c_{qj}) \vert i \in \mathcal{S}, j \in V_D, q \in \mathcal{N}^-(j)\}$ in $\mathcal{H}$ from all vertices $i \in \mathcal{S}$ to all destination vertices $j \in V_D$.

\subsection{Integrality}
\label{subsec: integrality}
Up to this point our CG algorithm does not yield integer solutions for \ref{eqn: IP 2}. To analyse large transportation system's optima and thus the potential of different transportation systems, finding fractional solutions often suffices. Even though integer solutions are not the main focus of our work, we can obtain integer solutions of good quality through a CG-based P\&B approach. In this approach, we apply the CG algorithm (Algorithm~\ref{ALG: 1}) once to find an optimal continuous solution and solve the resulting program with the new columns found via CG as an integer program (IP) with a standard IP solver. Note that we can now utilize our pricing filter (Algorithm~\ref{ALG: 2}) and the $A^*$-based pricing algorithm in Algorithm~\ref{ALG: 1}. This CG-based P\&B approach does not guarantee optimal integer solutions because the feasible region of the resulting IP is a subset of the feasible region of the original master problem after Algorithm~\ref{ALG: 1} terminates. Nevertheless, our CG-based P\&B approach allows us to find integer solutions of good quality quickly. From such an integer solution we can then construct a path for every passenger $p \in \mathcal{P}$ - if one exists - (see Section~\ref{subsec: example}) in which every path fulfills all the constraints defined in Section~\ref{sct: problem setting}.

Note that generally, the main drawback of a P\&B approach is, that the resulting IP can be infeasible \citep{sadykov2019primal}. To mitigate this drawback, we ensure that our CG-based P\&B approach always finds an integer solution because the variables used for initiating the CG algorithm are still contained in the IP and give a trivial integer solution.

\section{Case Study}
\label{sec: case study}
Our case study bases on a real-world setting for the city of Munich, Germany. We derive the public transport network consisting of bus, subway, and tram lines (see Figure~\ref{fig: transportation networks}), together with its schedules from GTFS data \citep{GTFS}. The bus network consists of 986 vertices and 2,228 arcs, while the subway network consists of 89 vertices and 184 arcs, and the tram network consists of 163 vertices and 338 arcs. We received information about vehicle capacity by the public transportation provider MVG \citep{MVG}. Here, busses have a capacity of 60, subways a capacity of 940 and trams a capacity of 215.
We obtained travel demand from the modeling tool MITO (cf. \citealp{moeckel2019microscopic}), which generates passenger data through a Monte-Carlo sampling followed by a nested mode choice model. We focus on three passenger datasets. The SUBWAY dataset consists of passengers only using the subway network with 2,206 trips between 7-9am, the BUS dataset consists of passengers only using the bus network with 26,320 trips between 7-9am and the BUS-SUBWAY-TRAM dataset consists of passengers using all possible intermodal combinations of the bus, subway, and tram network with 62,550 trips between 7-9am. 

\begin{figure}[t]
	\begin{subfigure}{.3\textwidth}
		\centering
		\includegraphics[scale=0.25]{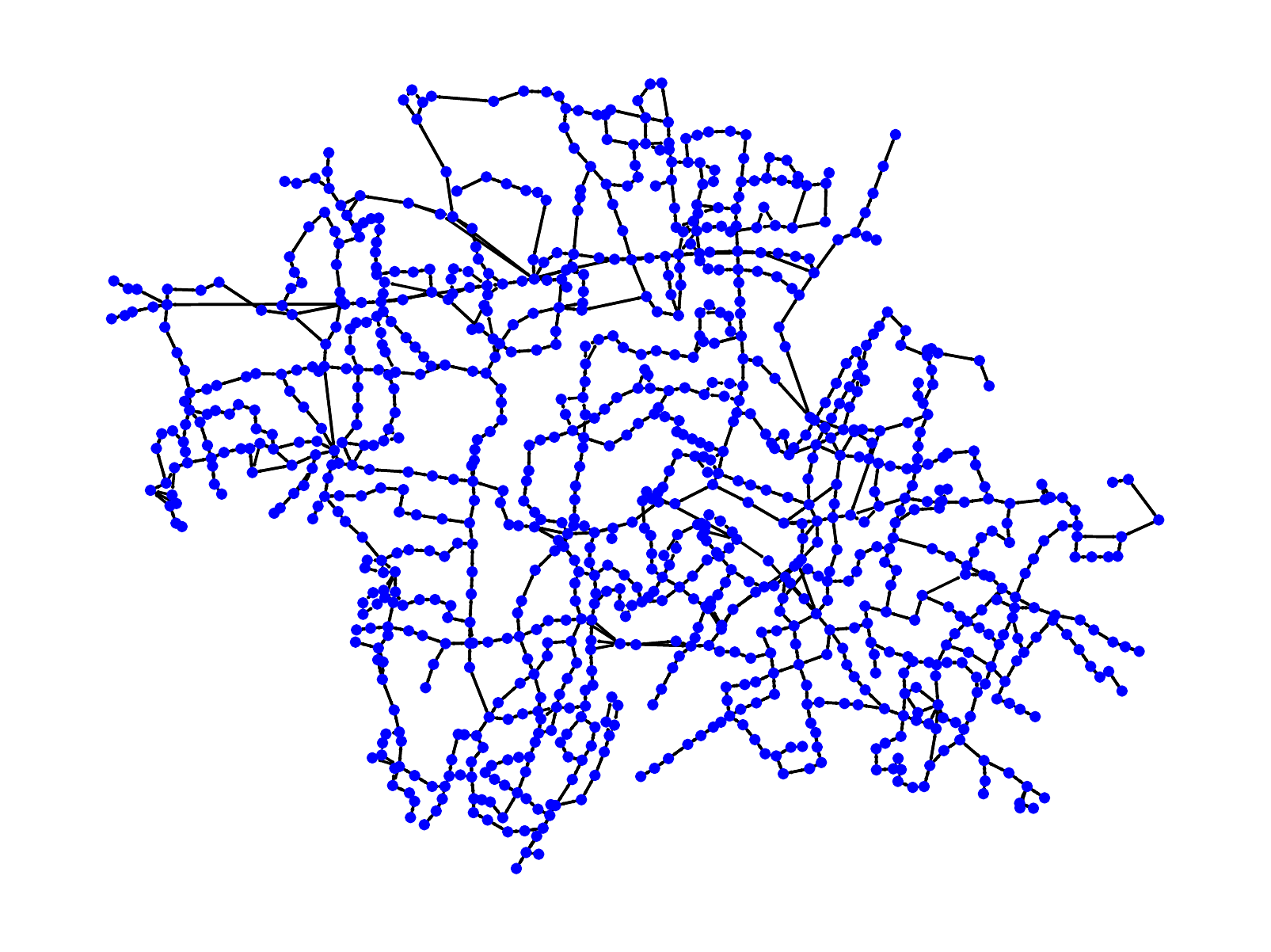}
		\caption{Bus network}
	\end{subfigure}
	\begin{subfigure}{.3\textwidth}
		\centering
		\includegraphics[scale=0.25]{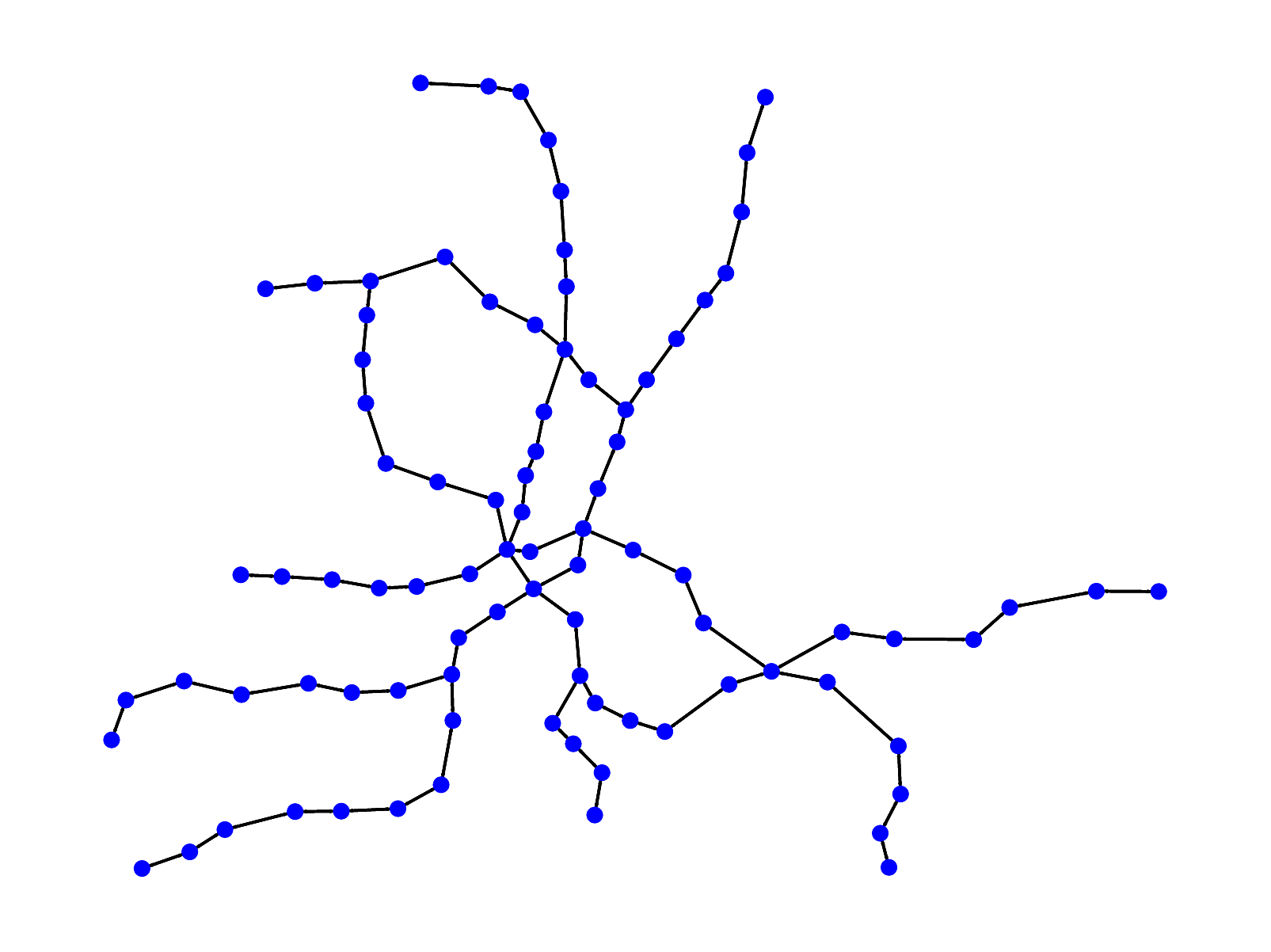}
		\caption{Subway network}
	\end{subfigure}
	\begin{subfigure}{.3\textwidth}
		\centering
		\includegraphics[scale=0.25]{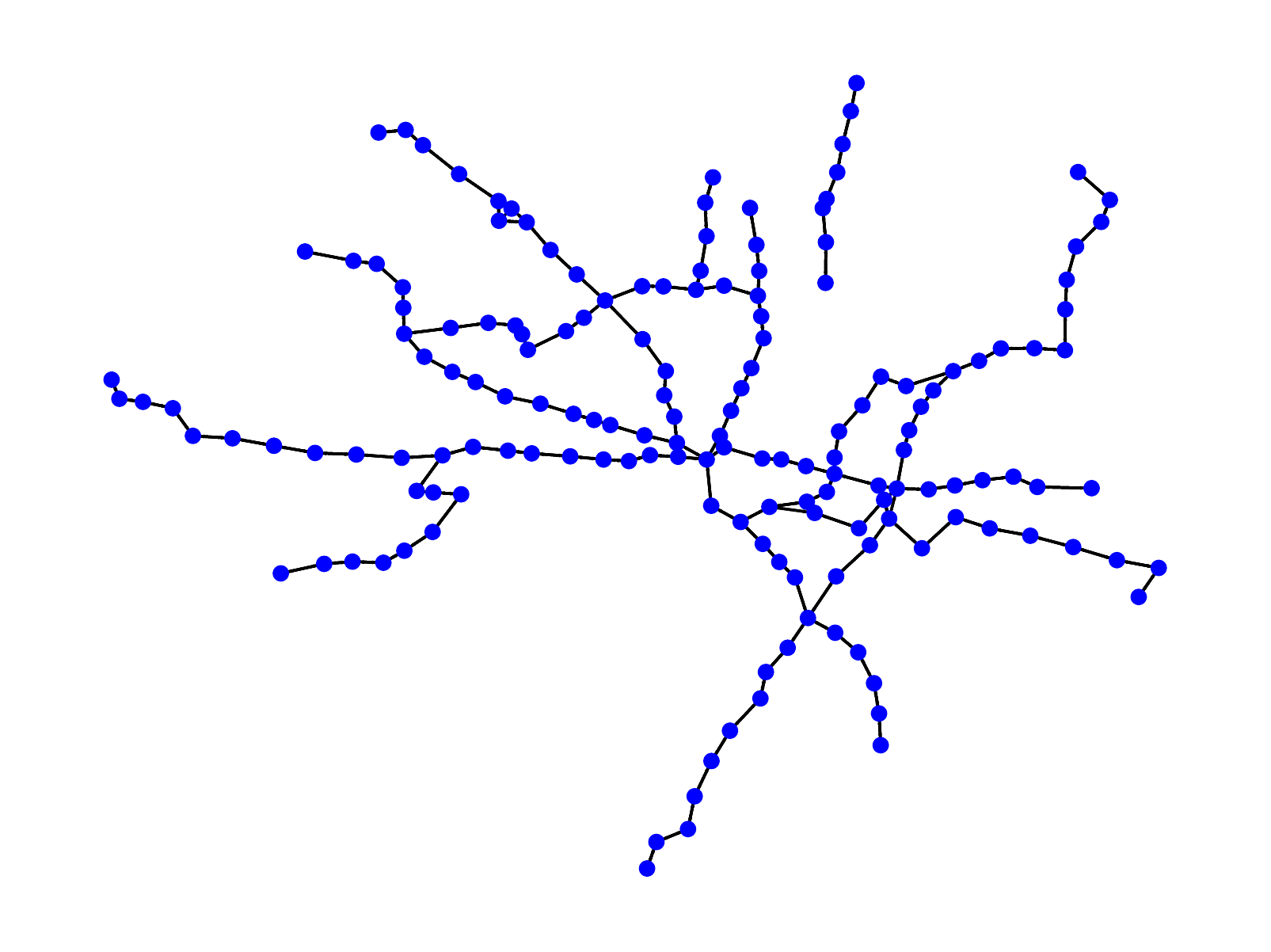}
		\caption{Tram network}
	\end{subfigure}
	\caption{Transportation networks of our Munich case study.png}
	\label{fig: transportation networks}
\end{figure}

We use this case study to analyse the effectiveness of i) our CG approach, ii) our pricing filter, iii) our $A^*$-based pricing algorithm, and iv) the quality of our integer solution, obtained by our CG-based P\&B approach. For analysis i) we generate instances with 132-662 passengers (6-30\%) from the SUBWAY dataset. For each passenger set size, we run 10 random instances, in which we scale the capacity of the vehicles according to the subset size, i.e., in an instance, with 10\% of all passengers of the dataset, only 10\% of the vehicles' capacity is available. For analysis ii) - iv), we generate instances with 2,632-23,688 passengers (10-90\%) for the BUS dataset and 6,255-56,295 (10-90\%) passengers for the BUS-SUBWAY-TRAM dataset. For each passenger set size, we again run 10 random instances, in which we scale the capacity of the vehicles according to the subset size. Note that we hold out 10\% of the passenger data during the instance generation to be able to generate at minimum 10 different instance for each instance size.

\section{Results}
\label{sec: results}
In this section we show the effectiveness of i) our CG approach, ii) our pricing filter, iii) our $A^*$-based pricing algorithm, and iv) the quality of our integer solutions based on the experiments described in Section \ref{sec: case study}. We set the maximum runtime of the algorithm to 60 minutes. All our experiments have been conducted on a standard desktop computer equipped with an \texttt{Intel(R) Core(TM) i9-9900, 3.1 GHz CPU} and \texttt{16 GB} of \texttt{RAM}, running \texttt{Ubuntu 20.04}. We have implemented the CG algorithm in \texttt{Python (3.8.11)} using \texttt{Gurobi 9.5} to solve the restricted master problem. Our source code as well as an overview of all results can be found on \url{https://github.com/tumBAIS/intermodalTransportationNetworksCG}.

\subsection{Effectiveness of the column generation}
Figure~\ref{fig: box cg-vs-ip} shows the computation time of our CG-based P\&B algorithm and \ref{eqn: IP 1} solved with a standard IP solver. Here, the CG-based P\&B algorithm does not utilize our pricing filter and solves the pricing problems with Dijkstra's algorithm. We see that our CG-based P\&B algorithm vastly outperforms a standard IP formulation even without any additional enhancements. For the IP formulation, we run out of memory for all instances with more than 662 passengers.
Table~\ref{tab: LP CG} complements this observation by showing the number of solved instances, construction time of the IP/$RMP$, the solving time, and the total time for instances with 132 - 662 passengers of the SUBWAY dataset. 
Here, the IP/$RMP$ construction time is the time needed to formulate \ref{eqn: IP 1} or the $RMP$ from Algorithm \ref{ALG: 1} respectively. The CG-based P\&B algorithm needs less than one second to find an integer solution while the IP formulation takes more than 9 minutes in the biggest subset size. The solving time in the CG-based P\&B algorithm is also much lower than in the IP formulation. With only 662 passengers, the standard IP formulation is not solvable for eight out of ten instances, because we run into memory bounds. 

\begin{figure}[!tb]
\centering
	\includegraphics[width=0.4\textwidth]{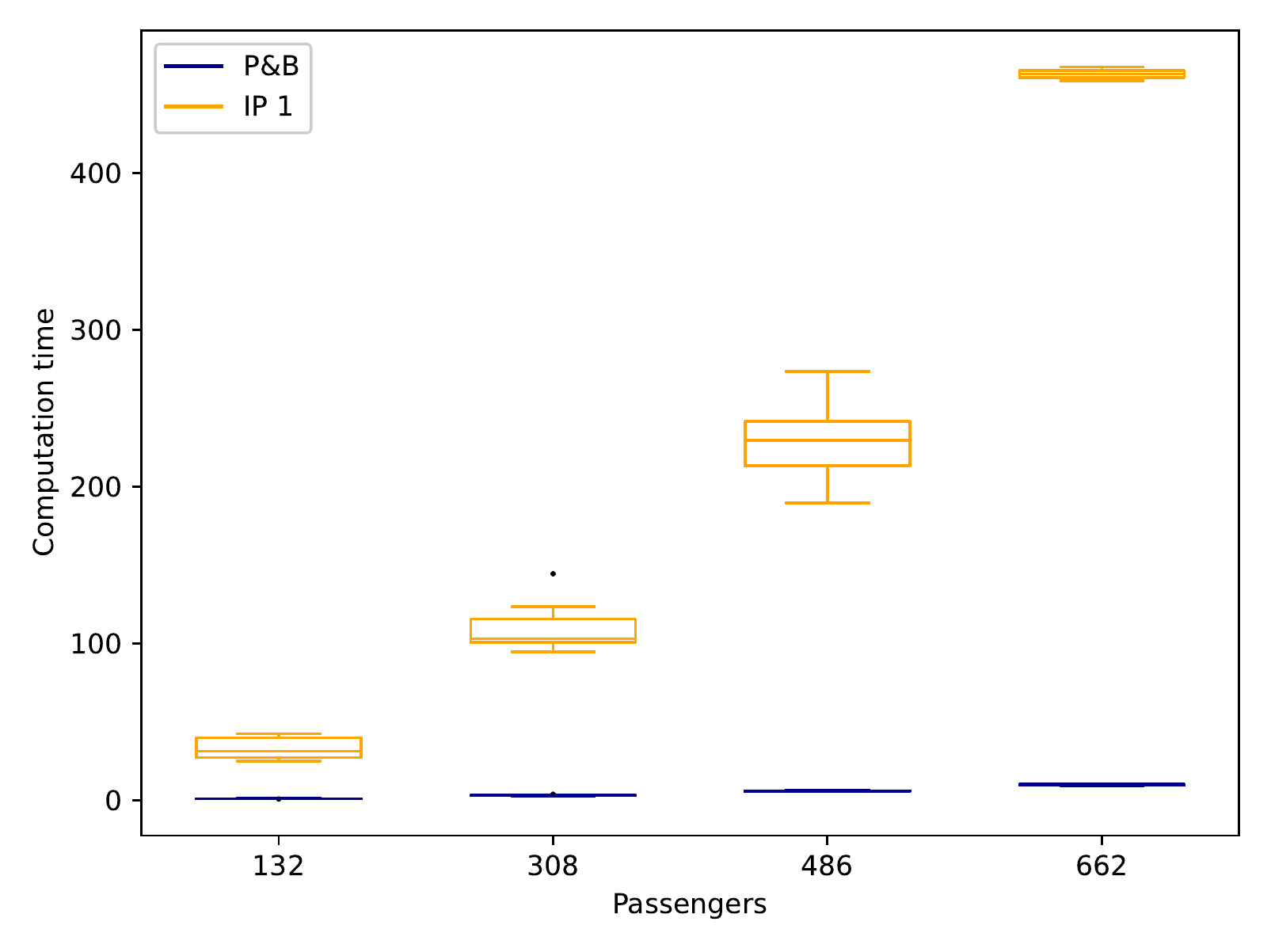}
	\caption{Total time comparison P\&B vs. IP 1 - SUBWAY}
	\label{fig: box cg-vs-ip}
\end{figure}

\setlength\tabcolsep{3.5pt}  % default: 6pt
\renewcommand{\arraystretch}{1.5}
\begin{table}[tb]
\centering
\scriptsize
\caption{Algorithm comparison \ref{eqn: IP 1} (IP) and Algorithm \ref{ALG: 1} without $A^*$ and filter (CG) - SUBWAY}
\begin{tabular}{@{} l*{8}{r} @{}}
\toprule
\textbf{Number of passengers} & \multicolumn{2}{c}{\textbf{132}} 
& \multicolumn{2}{c}{\textbf{308}} & \multicolumn{2}{c}{\textbf{486}} & \multicolumn{2}{c}{\textbf{662}}\\
\cmidrule(lr){1-1} \cmidrule(lr){2-3} \cmidrule(lr){4-5} \cmidrule(lr){6-7} \cmidrule(lr){8-9}  
Algorithm  & CG& IP & CG& IP & CG& IP & CG & IP\\
\midrule
\# Solved instances (out of 10) & 10 & 10 & 10 & 10 & 10 & 10 & 10 & 2 \\
Median IP/RMP construction time & $<$1s & 48s & $<$1s & 142s & $<$1s & 315s & $<$1s & 553s\\ 
Mean solving time & 1s & 33s & 3s & 109s & 6s & 228s & 10s & 463s\\
Mean total time & 1s& 82s & 3s & 252s & 6s & 542s & 10s & 1016s\\
\bottomrule
\end{tabular}
\label{tab: LP CG}
\end{table}

\subsection{Effectiveness of the pricing filter}
Table~\ref{tab: Filter bus} shows the number of passengers, solved instances, computation time, and number of solved pricing problems for the subset of all travel requests ranging from 2,632-23,688 passengers of the BUS dataset, computed with our CG-based P\&B algorithm utilizing and not utilizing our pricing filter, using the $A^*$-based pricing algorithm to solve each pricing problem.

All instances are solved to optimality with and without our pricing filter. However, when we apply our pricing filter the number of solved pricing problems decreases significantly by 60-65\%, which leads to a 35-50\% reduction in computation time. 
Figure~\ref{fig: pricing problems bus} and Figure~\ref{fig: Computation time bus} show the number of solved pricing problems and the computation time for the different instance sizes respectively. We see that our pricing filter leads to stable numbers of pricing problems and computation times.\medskip

Table~\ref{tab: Filter bus subway tram} extends our analysis to the BUS-SUBWAY-TRAM dataset. Here, 6,255 trips can be solved in one minute and 56,295 trips take 48 minutes to solve when using our pricing filter. For 43,785 passengers and beyond, no instance can be solved without using the pricing filter. In the instances that can be solved without our pricing filter, a computational speedup of up to 60\% can be achieved by applying our pricing filter. The reduction of 60\% in the number of solved pricing problems is similar to the reduction in the instances of the BUS dataset. Figure~\ref{fig: pricing problems bus subway tram} and Figure~\ref{fig: Computation time bus subway tram} again show the number of solved pricing problems and the computation time for the different instance sizes respectively. We see that similar to the BUS dataset our pricing filter leads to a stable number of pricing problems and to stable computation times. Concluding, our pricing filter leads to improved stable performances across different instance sizes as well as single and intermodal graph structures.

\begin{table}[!t]
\centering
\scriptsize
\caption{Algorithm comparison with and without the filter using $A^*$ - BUS dataset}
\hspace*{-8ex}
\begin{tabular}{@{} l*{10}{r} @{}}
\toprule
\textbf{Number of passengers} & \multicolumn{2}{c}{\textbf{2,632}}
& \multicolumn{2}{c}{\textbf{7,896 }}
& \multicolumn{2}{c}{\textbf{13,160}} & \multicolumn{2}{c}{\textbf{18,424}} & \multicolumn{2}{c}{\textbf{23,688}}\\
\cmidrule(lr){1-1} \cmidrule(lr){2-3} \cmidrule(lr){4-5} \cmidrule(lr){6-7} \cmidrule(lr){8-9} \cmidrule(l){10-11} 
Pricing filter (On/Off) & On & Off& On & Off & On & Off& On & Off& On & Off\\
\midrule
\# Solved instances (out of 10)  & 10 & 10 & 10 & 10 & 10 & 10 & 10 & 10 & 10 & 10 \\
Mean computation time & 33s & 51s & 107s & 192s & 245s & 420s & 382s & 737s & 603s & 1129s\\
Mean \# pricing problems & 9,606 & 24,491 & 23,450  & 62,405 & 34,956 & 97,418 & 46,143 & 132,699 & 59,372 & 165,887\\
\bottomrule
\end{tabular}
\label{tab: Filter bus}
\end{table}

\begin{figure}[!t]
\RawFloats
\captionsetup{justification=centerlast}
\begin{minipage}[t]{0.45\textwidth}
    \includegraphics[width=\textwidth]{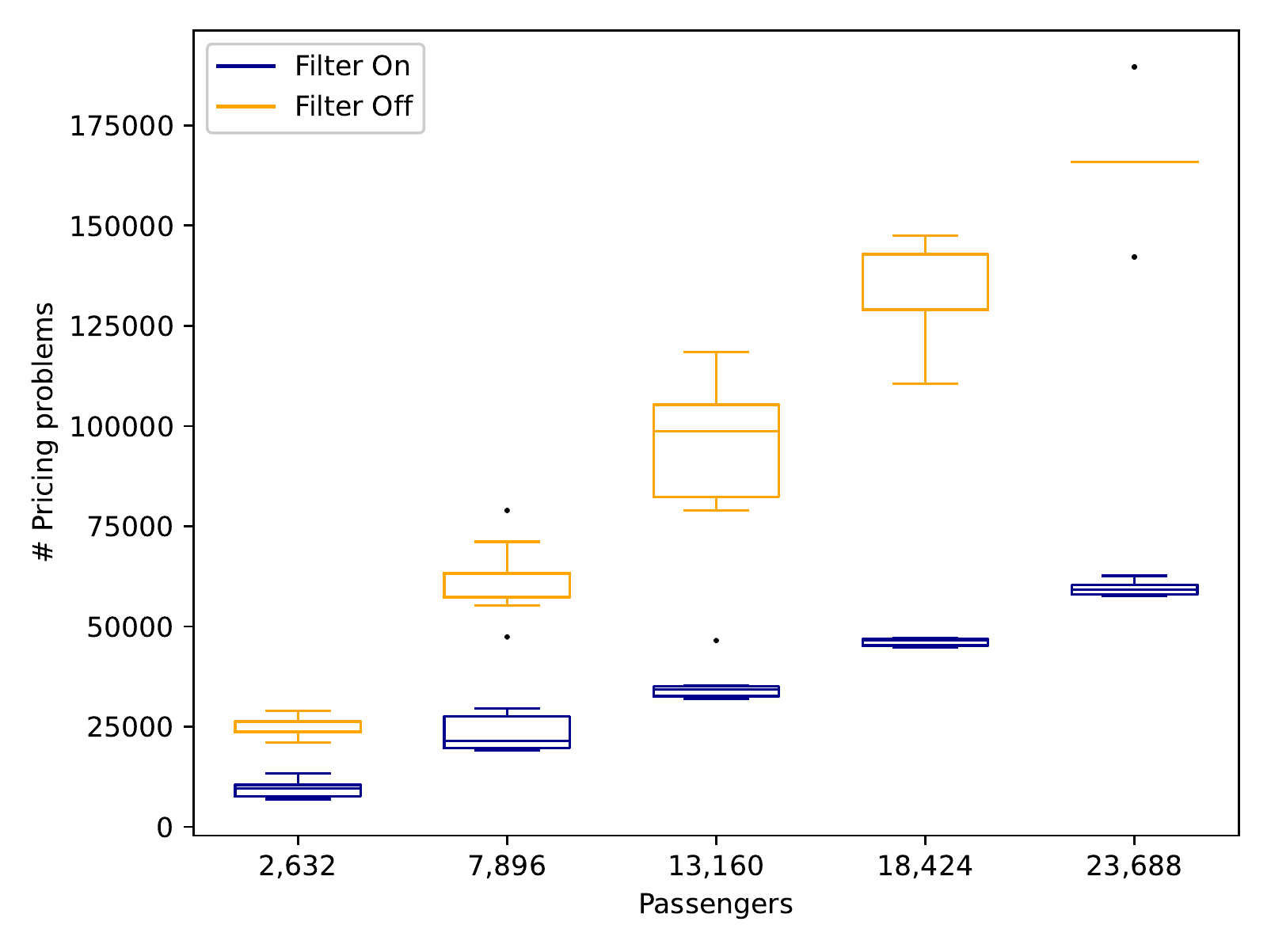}
    \caption{Number of pricing problems\\
		(BUS)}
    \label{fig: pricing problems bus}
\end{minipage}\hfill
\begin{minipage}[t]{0.45\textwidth}
    \includegraphics[width=\textwidth]{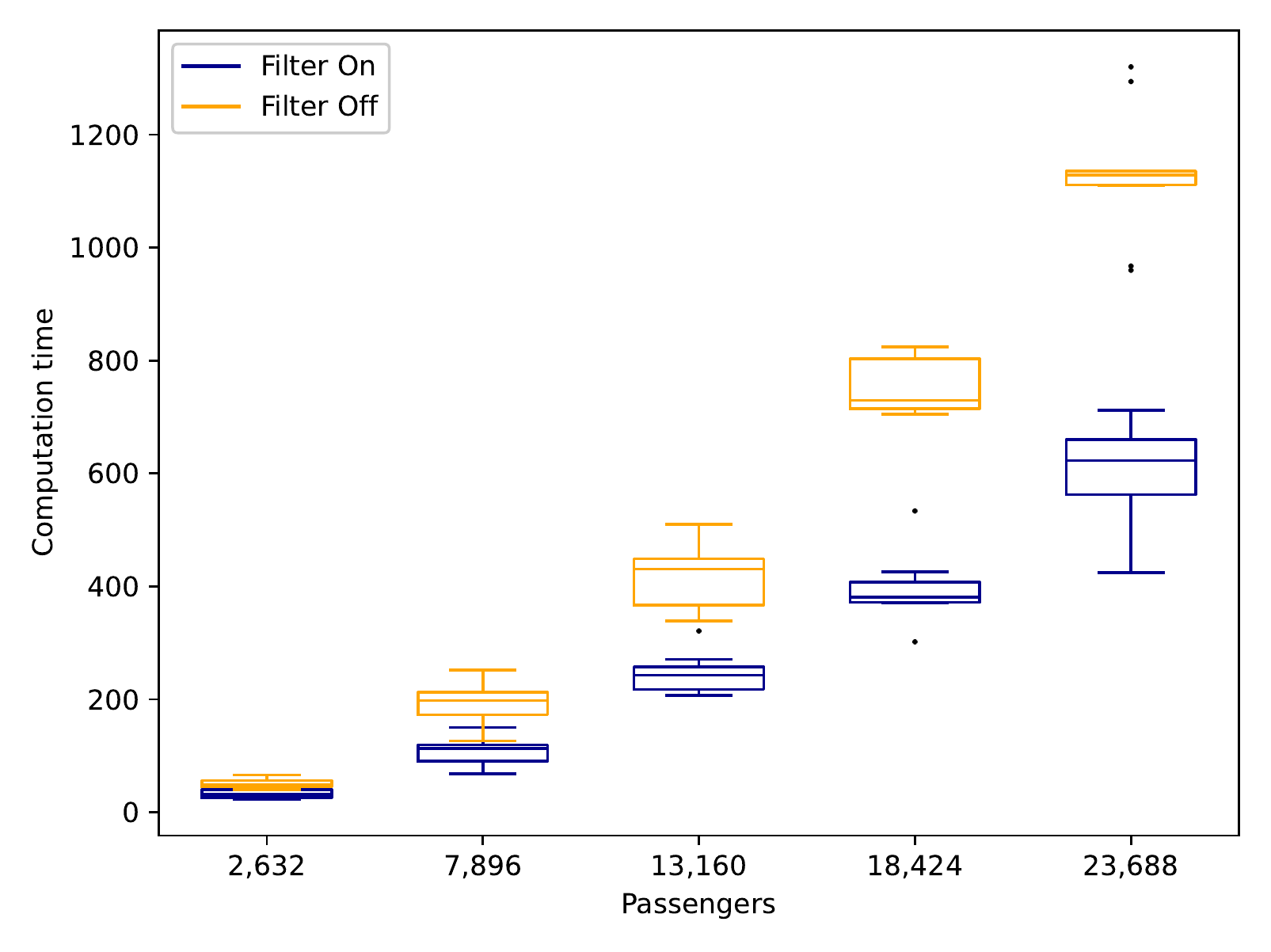}
    \caption{Computation time\\
		(BUS)}
    \label{fig: Computation time bus}
\end{minipage}
\end{figure}

\begin{table}[t]
\centering
\scriptsize
\caption{Algorithm comparison with and without the filter using $A^*$ - BUS-SUBWAY-TRAM dataset}
\hspace*{-8ex}
\begin{tabular}{@{} l*{10}{r} @{}}
\toprule
\textbf{Number of passengers} & \multicolumn{2}{c}{\textbf{6,255}}
& \multicolumn{2}{c}{\textbf{18,765}}
& \multicolumn{2}{c}{\textbf{31,275}} & \multicolumn{2}{c}{\textbf{43,785}} & \multicolumn{2}{c}{\textbf{56,295}}\\
\cmidrule(lr){1-1} \cmidrule(lr){2-3} \cmidrule(lr){4-5} \cmidrule(lr){6-7} \cmidrule(lr){8-9} \cmidrule(l){10-11} 
Pricing filter (On/Off) & On & Off& On & Off & On & Off& On & Off& On & Off\\
\midrule
\# Solved instances (out of 10)  & 10 & 10 & 10 & 9 & 10 & 8 & 10 & 0 & 10 & 0 \\
Mean computation time & 66s & 142s & 370s & 891s & 900s & 2243s & 1766s & - & 2856s & -\\
Mean \# pricing problems & 13,810 & 36,273 & 40,157  & 108,816 & 66,696 & 178,240 & 93,469 & - & 119,233 & -\\
\bottomrule
\end{tabular}
\label{tab: Filter bus subway tram}
\end{table}

\begin{figure}[t]
\RawFloats
\captionsetup{justification=centerlast}
\begin{minipage}[t]{0.45\textwidth}
    \includegraphics[width=\textwidth]{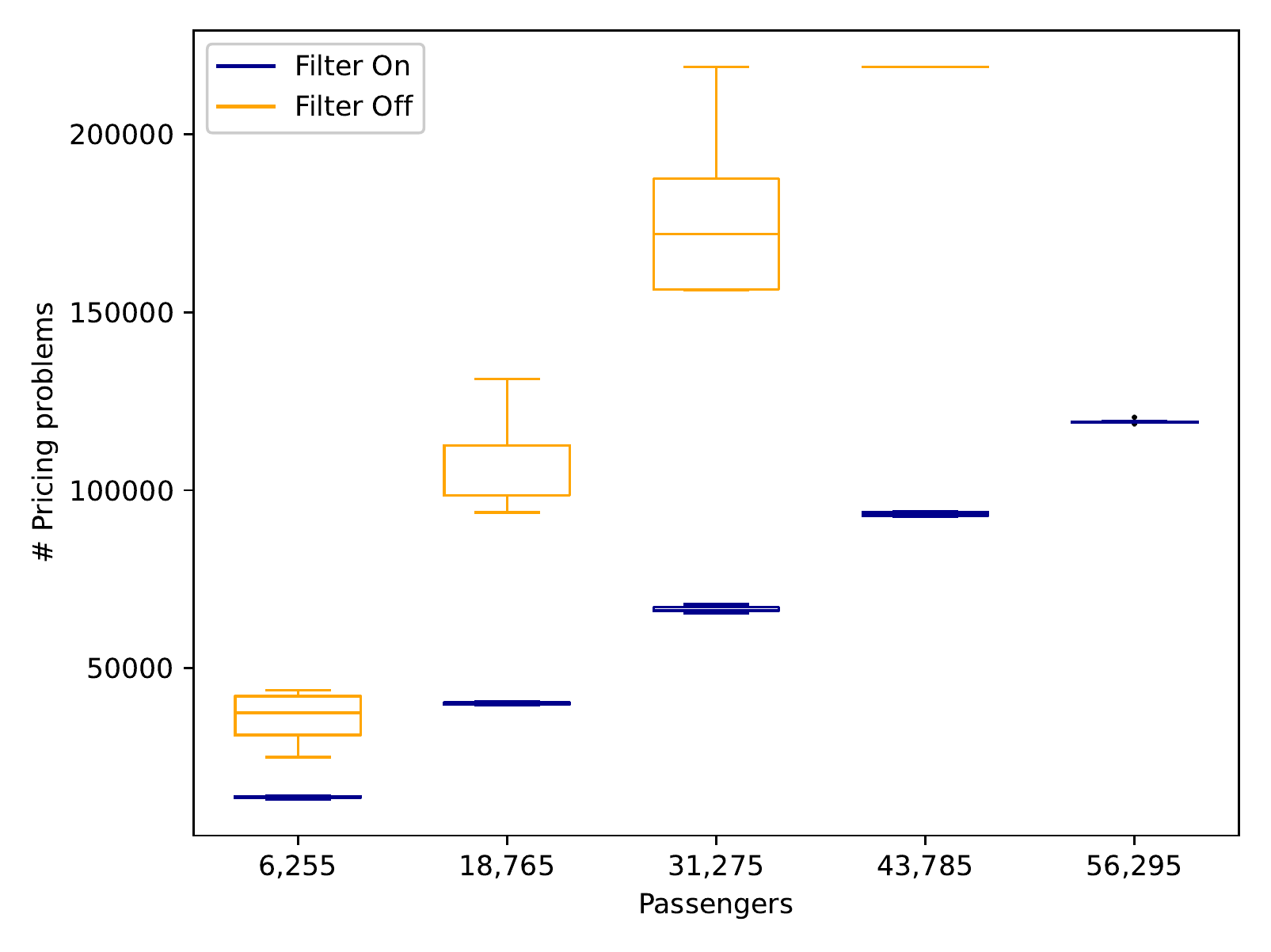}
    \caption{Number of pricing problems \\
		(BUS-SUBWAY-TRAM)}
    \label{fig: pricing problems bus subway tram}
\end{minipage}\hfill
\begin{minipage}[t]{0.45\textwidth}
    \includegraphics[width=\textwidth]{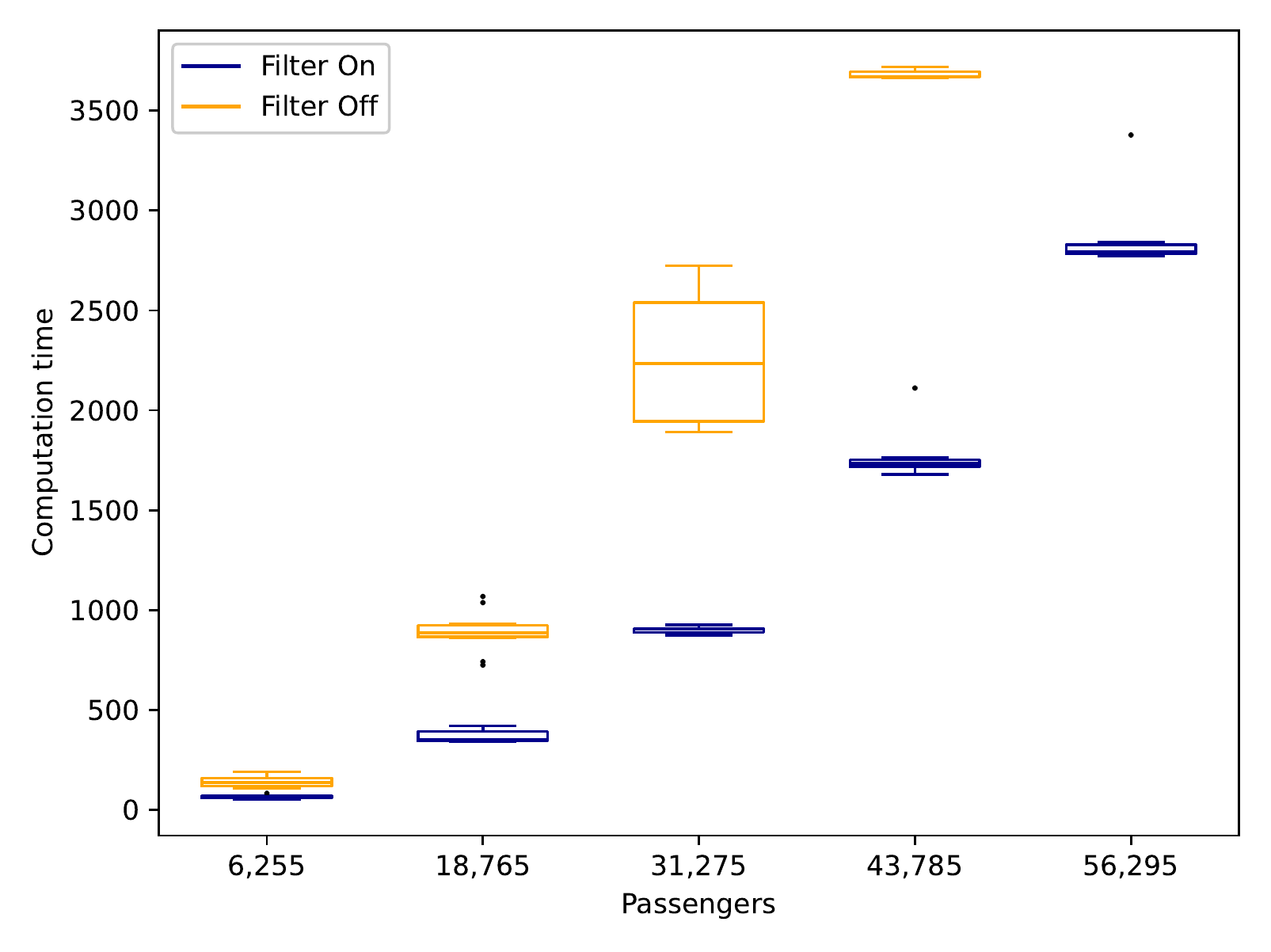}
    \caption{Computation time \\
		(BUS-SUBWAY-TRAM)}
    \label{fig: Computation time bus subway tram}
\end{minipage}
\end{figure}

\subsection{Effectiveness of the \texorpdfstring{$A^*$}-based pricing algorithm}
Table~\ref{tab: astar dijkstra bus} shows the number of passengers, solved instances, and the computation time of our CG-based P\&B algorithm with the $A^*$-based pricing algorithm and the Dijkstra-based pricing algorithm for the BUS dataset, utilizing our pricing filter in both algorithms. We see that the integration of the $A^*$-based pricing algorithm yields a computational speedup of up to 85\% in the BUS dataset. No instance with 23,688 passengers and only eight instances with 18,424 passengers can be solved within 60 minutes with the Dijkstra-based pricing algorithm, where on the other hand, all instances can be solved with the $A^*$-based pricing algorithm.

Table~\ref{tab: astar dijkstra bus subway tram} extends our analysis to the BUS-SUBWAY-TRAM dataset. For this dataset, we observe a similar speedup of the computation time as in the BUS dataset for smaller instances with 6,255 instances. Here, the $A^*$-based pricing algorithm leads to 90\% faster computation times than the Dijkstra-based pricing algorithm. In contrast to the BUS dataset, bigger instances with 18,765 passengers or more cannot be solved with the Dijkstra-based pricing algorithm, while all instances can be solved with the $A^*$-based pricing algorithm. Here, the graph constructed for the intermodal BUS-SUBWAY-TRAM transportation network is much bigger than the graph constructed for the BUS transportation network. Without the additional information provided through the distance approximation, the Dijkstra algorithm has to expand much more nodes to find the shortest path for a pricing problem. The bigger the underlying graph, the greater the difference between the number of the explored nodes between the Dijkstra-based and the $A^*$-based pricing algorithm gets. Figure~\ref{fig: astar vs dijkstra bus} and Figure~\ref{fig: astar vs dijkstra bus subway tram} complement this analysis and show that our algorithm with the $A^*$-based pricing algorithm leads to more stable computation times in both datasets compared to the Dijkstra-based pricing algorithm. 

We note that the distance approximation needed for the $A^*$-based pricing algorithm was calculated in under 17 seconds for all BUS instances and in under 45 seconds for all BUS-SUBWAY-TRAM instances. The significant improvement in computation time in the CG justifies this additional setup time.

\begin{table}[t]
\caption{Algorithm comparison $A^*$ vs. Dijkstra with the filter active - BUS}
\centering
\scriptsize
\hspace*{-8ex}\begin{tabular}{@{} l*{10}{r} @{}}
\toprule
\textbf{Number of passengers} & \multicolumn{2}{c}{\textbf{2,632}}
& \multicolumn{2}{c}{\textbf{7,896}}
& \multicolumn{2}{c}{\textbf{13,160}} & \multicolumn{2}{c}{\textbf{18,424}} & \multicolumn{2}{c}{\textbf{23,688}}\\
\cmidrule(lr){1-1} \cmidrule(lr){2-3} \cmidrule(lr){4-5} \cmidrule(lr){6-7} \cmidrule(lr){8-9} \cmidrule(l){10-11}
Pricing problem solver & $A^*$ & Dijkstra& $A^*$ & Dijkstra& $A^*$ & Dijkstra& $A^*$ & Dijkstra& $A^*$ & Dijkstra\\
\midrule
\# Solved instances (out of 10) & 10 & 10 & 10 & 10 & 10 & 10 & 10 & 8 & 10 & 0 \\
Mean computation time & 33s & 239s & 107s & 914s & 245s & 1932s & 382s & 3022s & 603s & -\\
\bottomrule
\end{tabular}
\label{tab: astar dijkstra bus}
\end{table}

\begin{table}[!t]
\caption{Algorithm comparison $A^*$ vs. Dijkstra with the filter active - BUS-SUBWAY-TRAM}
\centering
\scriptsize
\hspace*{-8ex}\begin{tabular}{@{} l*{10}{r} @{}}
\toprule
\textbf{Number of passengers} & \multicolumn{2}{c}{\textbf{6,255}}
& \multicolumn{2}{c}{\textbf{18,765}}
& \multicolumn{2}{c}{\textbf{31,275}} & \multicolumn{2}{c}{\textbf{43,785}} & \multicolumn{2}{c}{\textbf{56,295}}\\
\cmidrule(lr){1-1} \cmidrule(lr){2-3} \cmidrule(lr){4-5} \cmidrule(lr){6-7} \cmidrule(lr){8-9} \cmidrule(l){10-11}
Pricing problem solver & $A^*$ & Dijkstra& $A^*$ & Dijkstra& $A^*$ & Dijkstra& $A^*$ & Dijkstra& $A^*$ & Dijkstra\\
\midrule
\# Solved instances (out of 10) & 10 & 10 & 10 & 0 & 10 & 0 & 10 & 0 & 10 & 0 \\
Mean computation time & 66s & 793s & 370s & - & 900s & - & 1765s & - & 2856s & -\\
\bottomrule
\end{tabular}
\label{tab: astar dijkstra bus subway tram}
\end{table}

\begin{figure}[!t]
\RawFloats
\captionsetup{justification=centerlast}
\begin{minipage}[t]{0.45\textwidth}
    \includegraphics[width=\textwidth]{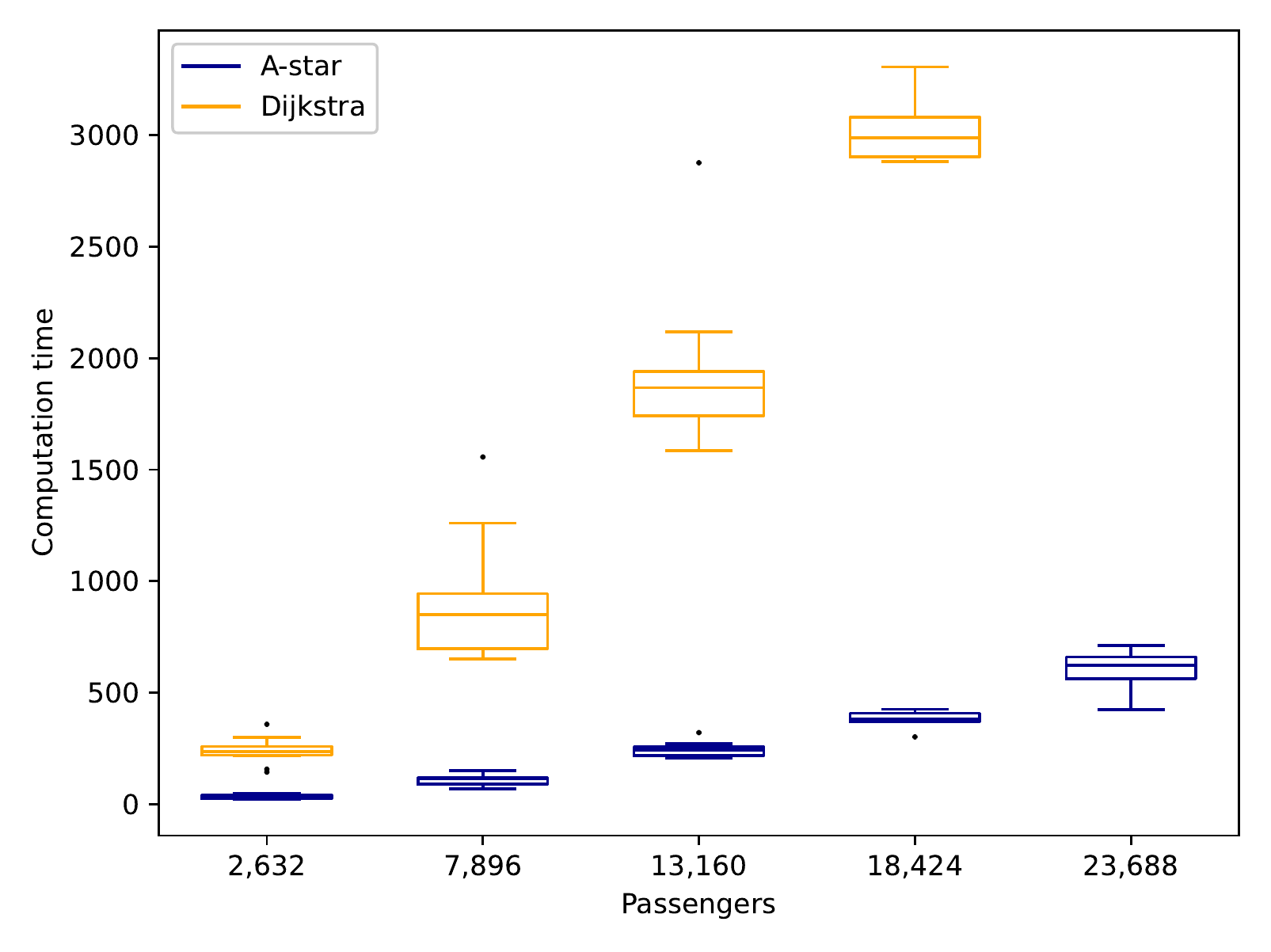}
    \caption{Computation time $A^*$ vs. Dijkstra \\
		(BUS)}
    \label{fig: astar vs dijkstra bus}
\end{minipage}\hfill
\begin{minipage}[t]{0.45\textwidth}
    \includegraphics[width=\textwidth]{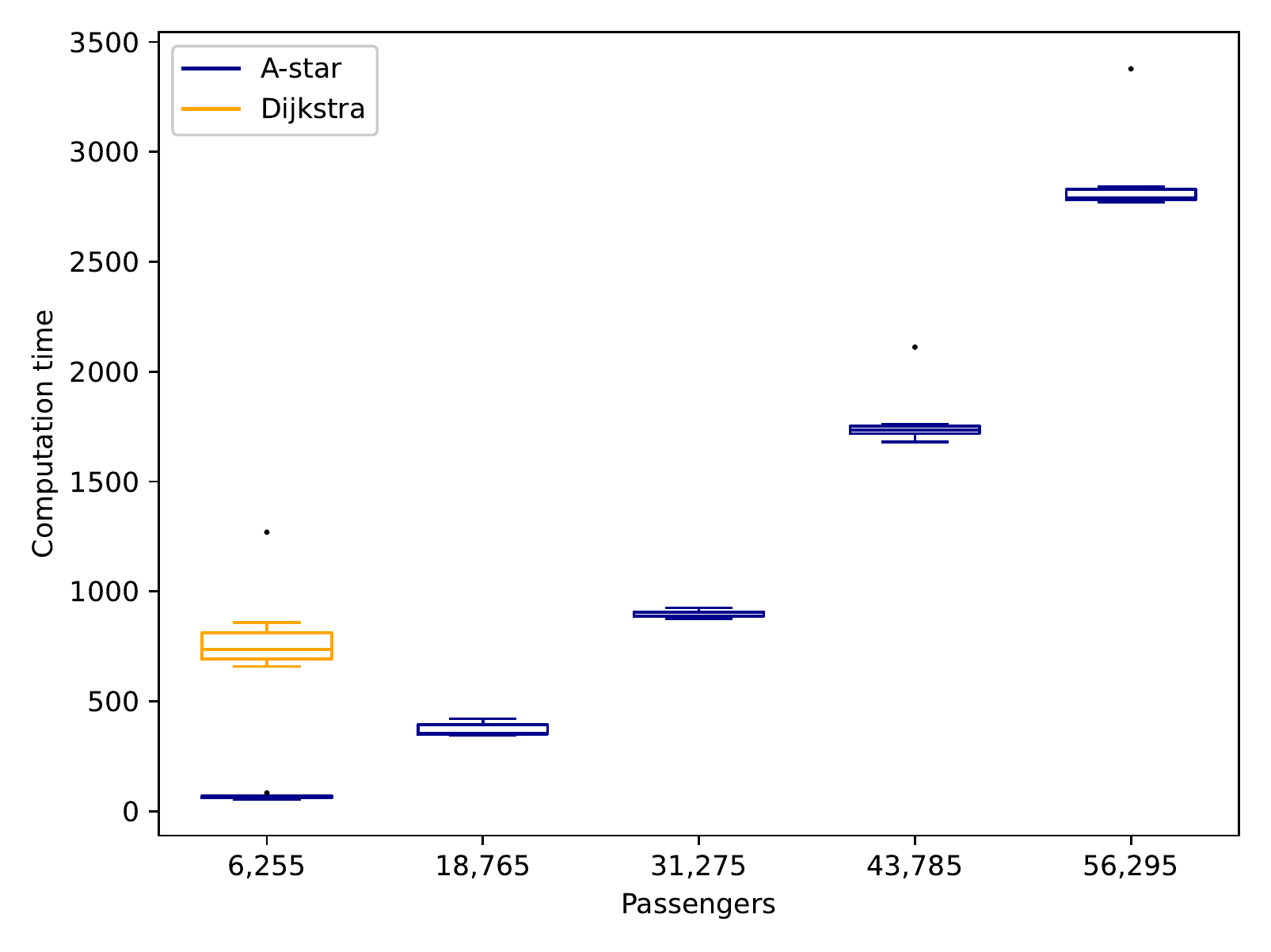}
    \caption{Computation time $A^*$ vs. Dijkstra \\
		(BUS-SUBWAY-TRAM)}
    \label{fig: astar vs dijkstra bus subway tram}
\end{minipage}
\end{figure}

\subsection{Integral solution quality}
Since finding integer solutions is not the main focus of this work, we used a standard IP solver to solve the IP introduced in our CG-based P\&B approach in Section \ref{subsec: integrality}. Table~\ref{tab: integral runtime bus} shows the number of solved instances, the total computation time, the computation time needed to find our integer solutions, and the optimality gap for instances of the BUS and BUS-SUBWAY-TRAM datasets, utilizing our CG-based P\&B algorithm including our pricing filter and the $A^*$-based pricing algorithm. Although our CG-based P\&B approach does not guarantee optimality for integer solutions, we are able to find optimal integer solutions for all instances in under five seconds. Here, we use the continuous solution of the CG algorithm as a lower bound. Furthermore, we see that finding integer solutions accounts only for a small percentage of the total computation time. Concluding, our CG-based P\&B approach leads to stable performances across different instance sizes as well as uni and intermodal transportation networks.

\begin{table}[tb]
\caption{Integral solutions}
\centering
\scriptsize
\hspace*{-8ex}\begin{tabular}{@{} l*{10}{r} @{}}
\toprule
& \multicolumn{5}{c}{\textbf{BUS}} & \multicolumn{5}{c}{\textbf{BUS-SUBWAY-TRAM}}\\ 
\cmidrule(lr){2-6} \cmidrule(lr){7-11}
\textbf{Number of passengers} & \textbf{2,632} & \textbf{7,896} & \textbf{13,160} & \textbf{18,424} & \textbf{23,688} & \textbf{6,255} & \textbf{18,765} & \textbf{31,275} & \textbf{43,785} & \textbf{56,295}\\
\midrule
\# Solved instances (out of 10) & 10 & 10 & 10 & 10 & 10 & 10 & 10 & 10 & 10 & 10 \\
Mean total computation time & 33s & 124s & 245s & 392s & 639s & 66s & 370s & 900s & 1766s & 2856s\\
Mean integer time & 1.3s & 1.7s & 2.0s & 2.1s & 2.7s & 1.8s & 4.0s & 5.0s & 3.9s & 4.6s\\
Mean optimality gap & 0.0\% & 0.0\% & 0.0\% & 0.0\% & 0.0\% & 0.0\% & 0.0\% & 0.0\% & 0.0\% & 0.0\%\\
\bottomrule
\end{tabular}
\label{tab: integral runtime bus}
\end{table}
\newpage

\section{Conclusion}
\label{sec: conclusion}

With this work, we introduced an algorithmic framework to determine the system optimum of an intermodal transportation system. To do so, we modeled a transportation system as a partially time-expanded digraph, which implicitly encodes problem specific constraints related to the intermodality of the problem. This allowed us to find the intermodal transportation system's optimum by solving a standard integer minimum-cost multi-commodity flow problem. We solved this problem with a column generation based price-and-branch approach, for which we introduced a pricing filter to reduce the number of solved pricing problems. Furthermore, we proposed a distance approximation to utilize the $A^*$ algorithm to solve the pricing problems more efficiently. We applied our methodology to a real-world case study for the city of Munich. Our results show that our pricing filter reduces the computation time of our algorithm by up to 60\% and the application of the $A^*$-based pricing algorithm additionally reduces it by up to 90\%. Our algorithm is able to solve unimodal instances with 23,688 trips in 11 minutes and intermodal instances with 56,295 trips in 48 minutes, whereas straightforward MIP models fail to solve instances with more than a few hundred trips.

%\onehalfspacing

%
%% finally the bib file
\singlespacing{
%\footnotesize
\bibliographystyle{model5-names}%\biboptions{authoryear}
\bibliography{./ms}} % if more than one, comma separated
\newpage
%% and the appendices
\onehalfspacing
\begin{appendices}
	\normalsize
	% Include any appendices here
\end{appendices}
\end{document}